\documentclass[12pt,reqno]{article}
\usepackage{amssymb}
\usepackage[mathscr]{eucal}
\usepackage[mathscr]{eucal}
\usepackage{amsmath,amssymb,latexsym,theorem,enumerate,ifthen}
\usepackage{orcidlink}
 \usepackage{hyperref} 

\newcommand{\DD}{\mathbb{D}}

\newcommand{\NN}{\mathbb{N}}

\newcommand{\RR}{\mathbb{R}}

\newcommand{\ZZ}{\mathbb{Z}}

\newcommand{\be}{{\boldsymbol{e}}}

\newcommand{\cA}{{\mathcal A}}
\newcommand{\cB}{{\mathcal B}}

\newcommand{\ee}{\mathrm{e}}

\newcommand{\EE}{\operatorname{\mathbb{E}}}

\newcommand{\PP}{\operatorname{\mathbb{P}}}

\newcommand{\cov}{\operatorname{Cov}}
\newcommand{\corr}{\operatorname{Corr}}
\newcommand{\sign}{\operatorname{sign}}

\newcommand{\comment}[1]{}

\renewcommand{\leq}{\leqslant}
\renewcommand{\geq}{\geqslant}

\newcommand{\proofend}{\hfill\mbox{$\Box$}}

\numberwithin{equation}{section}

\theoremstyle{change} \theorembodyfont{\em}
\newtheorem{Lem}{Lemma.}[section]
\newtheorem{Thm}[Lem]{Theorem.}
\newtheorem{Pro}[Lem]{Proposition.}

\newtheorem{Cor}[Lem]{Corollary.}

\theorembodyfont{\rm}
\newtheorem{Def}[Lem]{Definition.}
\newtheorem{Rem}[Lem]{Remark.}

\newtheorem{Ex}[Lem]{Example.}

\def\OnlyOnArXiv#1#2{\ifthenelse{\equal{#1}{Y}}{#2}{}}

\def\eq#1{{\rm(\ref{#1})}}
\long\def\Eq#1#2{\ifthenelse{\equal{#1}{*}}
  {\begin{equation*}\begin{aligned}#2\end{aligned}\end{equation*}}
  {\begin{equation}\begin{aligned}\label{#1}#2\end{aligned}\end{equation}}}

\newenvironment{proof}{\noindent{\bf Proof.}}{\proofend}

\begin{document}

\begin{center}
 {\bfseries\Large Expected loss of quasi-arithmetic means\\[1mm] of exchangeable random variables}

\vspace*{3mm}

{\sc\large
  M\'aty\'as $\text{Barczy}^{*,\diamond,
  \orcidlink{0000-0003-3119-7953}}$,
  Zsolt $\text{P\'ales}^{**,\orcidlink{0000-0003-2382-6035}}$ }

\end{center}

\vskip0.2cm

\noindent
 * HUN-REN–SZTE Analysis and Applications Research Group,
   Bolyai Institute, University of Szeged,
   Aradi v\'ertan\'uk tere 1, H--6720 Szeged, Hungary.

\noindent
 ** Institute of Mathematics, University of Debrecen,
    Pf.~400, H--4002 Debrecen, Hungary.

\noindent E-mails: barczy@math.u-szeged.hu (M. Barczy),
                  pales@science.unideb.hu  (Zs. P\'ales).

\noindent $\diamond$ Corresponding author.

\vskip0.2cm


{
\renewcommand{\thefootnote}{}
\vspace*{0.2cm}
}

\vspace*{-10mm}

\begin{abstract}
The purpose of this paper is to establish certain relationships between the theories of means (aggregation functions) and loss functions.
Namely, we derive sufficient conditions under which the sequence of the expected losses of a quasi-arithmetic mean
of the first $n$ members from a sequence of exchangeable random variables is (strictly) decreasing with respect to $n$.
Some new properties of functions which are strongly convex with respect to a nonnegative even error function, such as a Jensen-type inequality, are described. Additionally, non-trivial new examples for such functions are presented as well.
\end{abstract}


\section{Introduction}
\label{section_intro}

The theory of means, which form an essential subclass of aggregation functions, is an old and important field of classical analysis, and it has a large number of applications in every branch of mathematics. For a recent survey on this matter, see Beliakov et al.\ \cite{BelBusCal16}.
In particular, it has applications in insurance mathematics and in the theory of utility functions as well.
Among others, Chuzdiak \cite{Chu} showed that the zero utility principle in insurance mathematics is a particular case of the so-called quasi-deviation means, which enabled him to characterize some properties of this principle such as comparison, equality or subadditivity.
Dealing with various aspects of buying and selling prices for lotteries, Chudziak and Chudziak \cite{ChuChu} showed that the concepts of the so-called willingness to accept and willingness to pay can also be described in terms of quasi-deviation means.

In the present paper, we establish some connections between the theories of means and loss functions, and further, our paper connects them with the statistical prediction theory, as well.
In statistics, one can combine the predictions of several base models using means (aggregation functions). In particular, the arithmetic mean, which is one of the simplest means, is frequently used.
The heuristic idea behind aggregation is that groups are collectively wiser than individuals, and there are several phrases addressing this idea: ''wisdom of crowds'', ''vox populi'', ''miracle of aggregation'' or ''two heads are better than one''.
In our paper, roughly speaking, we investigate certain relationships between (strongly) convex loss functions and predictions based on the quasi-arithmetic means of exchangeable random variables.
(Note that arithmetic mean mentioned before is a special quasi-arithmetic mean.)
In the subsequent paragraphs, we introduce some notations, the notion of quasi-arithmetic means, and then, after Definition \ref{Def_quasi_arithmetic}, we formulate the problem presented above in mathematical terms.

Throughout this paper, let $\NN$, $\ZZ_+$, $\RR$ and $\RR_+$ denote the sets of positive integers, non-negative integers, real numbers and  nonnegative real numbers, respectively.
An interval $I\subseteq\RR$ will be called nondegenerate if it contains at least two distinct points.
Given a nondegenerate interval $I\subseteq\RR$, the difference $I-I$ stands for the set $\{x-y : x,y\in I\}$.
To avoid misunderstandings, in some cases we write $(I-I)$ instead of $I-I$.

Next, we recall the notion of quasi-arithmetic means.

\begin{Def}[Quasi-arithmetic mean]\label{Def_quasi_arithmetic}
Let $n\in\NN$, let $I\subseteq \RR$ be a nonempty open interval, and let $f:I\to\RR$ be a continuous and strictly increasing function.
The $n$-variable quasi-arithmetic mean $\mathscr{A}_n^{f}:I^n\to I$ is defined by
 \[
  \mathscr{A}_n^{f}(x_1,\ldots,x_n):= f^{-1} \bigg( \frac{1}{n} \sum_{i=1}^n f(x_i)\bigg), \qquad x_1,\ldots, x_n\in I,
 \]
where $f^{-1}$ denotes the inverse of $f$. The function $f$ is called the generator of $\mathscr{A}_n^{f}$.
\end{Def}

The arithmetic, geometric and harmonic means are quasi-arithmetic means corresponding to the generator \ $f:\RR\to \RR$, $f(x):=x$, $x\in\RR$;
 \ $f:(0,\infty)\to \RR$, \ $f(x):=\ln(x)$, \ $x>0$; \ and \ \ $f:(0,\infty)\to \RR$, \ $f(x) =x^{-1}$, \ $x>0$, \ respectively.

We will study the following problem.
Given a nonempty open interval $I\subseteq \RR$, a continuous and strictly increasing function $f:I\to\RR$, a (loss) function $L:I\to\RR$, and a sequence of random variables $(\xi_\ell)_{\ell\in\NN}$ 
 with values in $I$, we aim to establish conditions under which the inequality
 \[
    \EE\big( L( \mathscr{A}_n^{f}(\xi_1,\ldots,\xi_n)\big) \leq \EE\big( L(\mathscr{A}_{n-1}^{f}(\xi_1,\ldots,\xi_{n-1})) \big)
 \]
 holds for each $n\geq 2$, $n\in\NN$.
This problem has been motivated by a corresponding result on the arithmetic means of exchangeable random variables
and convex loss functions due to Marshall and Proschan \cite{MarPro} (see also Marshall et al.\ \cite[Proposition B.2.b on page 395]{MarOlk}), for more details, see Section \ref{Sec_convex_loss}.
Mattei and Garreau \cite[Section 6.2]{MatGar} initiated to consider other means than the arithmetic mean.

In Section \ref{Sec_convex}, the notions and some results on convex functions, strongly convex functions with nonnegative modulus, and
strongly convex functions with respect to a nonnegative even error function are recalled, see Definitions \ref{Def_convexity}, \ref{Def_strongly_convex_3} and \ref{Def_strongly_convex_2}, Theorem \ref{Thm_strongly_convex}, Proposition \ref{Pro_basic_properties}, Theorem \ref{Thm_MakNikPal} and Corollary \ref{Cor_MakNikPal1}, respectively.
In Proposition \ref{Prop_diff_strongly_convex}, we establish a particular case of part (ii) of Theorem 8 in Mak\'o et al.\ \cite{MakNikPal}, namely, we give a characterisation of differentiable strongly convex functions
 with respect to a nonnegative even error function.
Example \ref{Ex_power_strongly_convex} is devoted to present a nontrivial class of strongly convex functions
 with respect to the error function $\RR\ni t\mapsto \vert t\vert^p$, where $p\in[2,\infty)$, namely,
 the class of functions $\big\{ (0,\infty) \ni x \mapsto x^p : p\in[2,\infty)\big\}$.
In Example \ref{Ex_power_strongly_convex_2}, we give another nontrivial example for a strongly convex function with respect to a nonnegative even error function.
Remark \ref{Rem_starlikeness} is devoted to establish some properties of the function defined in \eqref{def_alphahat}, which plays an important role in the theory of strongly convex functions with respect to a nonnegative even error function, and this function turns out to be so-called starlike with respect to the point $0$.
Further, we prove a Jensen-type inequality for strongly convex functions with respect to a nonnegative even error function, see Corollary \ref{Cor_strong_alpha_convex}.
 
In Section \ref{Sec_Mat_Gar}, we recall the notion and some basic properties of exchangeability of random variables,
see Definition \ref{Def_Exch} and Proposition \ref{Pro1}.
Furthermore, motivated by an example due to Mattei and Garreau \cite[Section 3.3]{MatGar}, we also present an example in which we demonstrate that the convex loss of a forecasting based on exchangeable random variables with non-equal weights may get worse in average compared to the one with equal weights (in fact, our example is also a generalization of the above mentioned one of Mattei and Garreau \cite[Section 3.3]{MatGar}), see Example \ref{Ex_MatGar}.

In Section \ref{Sec_convex_loss}, first, we recall a result on arithmetic means of exchangeable random variables and convex loss functions due to Marshall and Proschan \cite{MarPro} (see also Marshall et al.\ \cite[Proposition B.2.b, page 395]{MarOlk} or Theorem \ref{Thm_Mat_Gar_1}).
In Theorem \ref{Thm_monotonicity}, we present a consequence of this result for quasi-arithmetic means. 
Roughly speaking, Theorem \ref{Thm_monotonicity} states that, given a nonempty open interval $I$, 
 a continuous and strictly increasing function $f$ on $I$, and a loss function $L$ on $I$ such that $L\circ f^{-1}$ is convex on $f(I)$, 
 the sequence of the expected losses of the quasi-arithmetic means of the first $n$ terms $\xi_1,\ldots,\xi_n$ of an exchangeable sequence of random variables 
 $(\xi_n)_{n\in\NN}$ with values in $I$ is decreasing with respect to $n$, where the quasi-arithmetic mean in question has generator $f$.
Under finiteness and positiveness of the variance of $f(\xi_1)$, in Theorem \ref{Thm_strict_monotonicity} and in Remark \ref{Rem_strongly_convex_mod_loss}, we generalize Theorem \ref{Thm_monotonicity} for strongly convex loss functions with a nonnegative modulus.
In this case, it turns out that the sequence of the expected losses in question modified by a certain correction term (depending on the modulus of the strongly convex function) is decreasing.
This result is also a generalization of Theorem 4 in dimension 1 of Mattei and Garreau \cite{MatGar} from arithmetic means to quasi-arithmetic means.
We close Section \ref{Sec_convex_loss} with Example \ref{Ex_Sect5}, in which we specialize Theorem \ref{Thm_strict_monotonicity} to the case of the geometric mean and a sequence of independent and identically 
 distributed (i.i.d.) random variables, and, in the second part of the example, we make a further specialization by choosing the uniform distribution on $(0,1)$ as the common distribution 
 of the i.i.d.\ sequence of random variables in question.
 
In Section \ref{Sec_convex_loss_3}, we obtain results that are certain extensions to those of in Section \ref{Sec_convex_loss}. Instead of strongly convex functions with a nonnegative modulus, we consider functions which are strongly convex functions with respect to a nonnegative even error function, see Theorem \ref{Thm_monotonicity_alpha}.
We remark that, in part (iii) of Theorem \ref{Thm_monotonicity_alpha}, we also provide a set of sufficient conditions under which the sequence $\big(\EE\big(L\big( \mathscr{A}_n^{f}(\xi_1,\ldots,\xi_n) \big)\big)\big)_{n\in\NN}$ is strictly decreasing.
In the proof of Theorem \ref{Thm_monotonicity_alpha}, we effectively use a Jensen-type inequality for strongly convex functions with respect to a nonnegative even error function (see Corollary \ref{Cor_strong_alpha_convex}).
Remarks \ref{Rem_condiii} and \ref{Rem_condiv_v} are devoted to shed some light on the assumptions in part (iii) 
 of Theorem \ref{Thm_monotonicity_alpha}.
In particular, we give an example of exchangeable (hence identically distributed) random variables $\xi_1,\xi_2$ and $\xi_3$ that are pairwise nonpositively correlated, but not independent.
In Corollary \ref{Cor_x4}, we specialize part (i) of Theorem  \ref{Thm_monotonicity_alpha} to the error function
 $\alpha(x):=x^4$, $x\in\RR$, and to a sequence of i.i.d.\ random variables. 
In this special case, the correction term in the inequality \eqref{help6} possesses a simpler form.
We close Section \ref{Sec_convex_loss_3} with Example \ref{Ex_Sect5_uj}, in which we specialize Corollary \ref{Cor_x4} to the case of geometric mean  and a sequence of i.i.d.\ random variables, and, in the second part of the example, we make a further specialization by choosing the uniform distribution on $(0,1)$ as the common distribution of the i.i.d.\ sequence of random variables in question.

\section{Preliminaries on convex and strongly convex functions}\label{Sec_convex}

First, we recall the notions of convex functions and strongly convex functions with a nonnegative modulus.

\begin{Def}\label{Def_convexity}
Let $I\subseteq \RR$ be a nondegenerate interval.
 A function $L:I\to\RR$ is called convex if
 \begin{align}\label{def_convex}
   L(tx+(1-t)y) \leq tL(x) + (1-t)L(y), \qquad t\in[0,1],\; x,y\in I.
 \end{align}
Further, $f$ is called strictly convex if strict inequality holds in \eqref{def_convex} for all $t\in(0,1)$ and $x\ne y$, $x,y\in I$.
\end{Def}

\begin{Def}\label{Def_strongly_convex_3}
Let $I\subseteq \RR$ be a nondegenerate interval and $\mu\in\RR_+$.
A function $L:I\to\RR$ is called strongly convex with modulus $\mu$ if
 \begin{align*}
  L(tx+(1-t)y) \leq t L(x) + (1-t)L(y) - \mu \cdot t (1-t) (x-y)^2,\qquad t\in[0,1], \; x,y\in I.
 \end{align*}
\end{Def}

For example, $L:\RR\to \RR$, $L(x):=x^2$, $x\in \RR$, is strongly convex with modulus $1$, since
 \begin{align}\label{convex_2}
   (tx+(1-t)y)^2 = tx^2 + (1-t)y^2 - t(1-t)(x-y)^2,\qquad t\in[0,1], \;\; x,y\in \RR.
 \end{align}
Note that if $L:I\to\RR$ is strongly convex with modulus $\mu\in\RR_+$, then $L$ is convex, and, 
 further, using \eqref{convex_2}, one can check that
 the function $\widetilde L: I\to\RR$ defined by $\widetilde L(x):=L(x) - \mu x^2$, $x\in I$, is convex as well.
Further, it is known that, given a twice differentiable function $L:(a,b)\to\RR$, where $a<b$, $a,b\in\RR$, we have that $L$ is strongly convex with modulus $\mu\in\RR_+$ if and only if $L''(x)\geq 2\mu$, $x\in(a,b)$, see, e.g., Roberts and Varberg \cite[page 268]{RobVar}. 
As a consequence, for example, if $\mu\in\RR_+$, $\kappa,\delta\in\RR$ and $L:\RR\to\RR$, $L(x):=\mu x^2+ \kappa x+\delta$, $x\in\RR$, then $L$ is strongly convex with modulus $\mu$, since $L''(x)=2 \mu$, $x\in\RR$. 
One can also see that the function $L=\cosh$ is strongly convex over $\RR$ with modulus $1/2$, since $\cosh''=\cosh\geq 1$ over $\RR$. The function $L=-\cos$ is convex over the interval $(-\pi/2,\pi/2)$, but it is not strongly convex with any positive modulus over this interval. On the other hand, if $-\pi/2<a<b<\pi/2$, then $L=-\cos$ is strongly convex over $(a,b)$ with modulus $\tfrac12\min(\cos(a),\cos(b))$, since 
$(-\cos)''(x)=\cos(x)\geq\min(\cos(a),\cos(b))$ for $x\in(a,b)$.

For strongly convex functions with a nonnegative modulus, the following strengthening of Jensen's inequality holds, see, e.g., Nikodem \cite[Theorem 2]{Nik}.

\begin{Thm}\label{Thm_strongly_convex}
Let $I\subseteq \RR$ be a nondegenerate interval, $\mu\in\RR_+$, and $L: I\to\RR$ be a strongly convex function with modulus $\mu$.
Then, for each $n\in\NN$ and $x_1,\ldots,x_n\in I$, we have that
 \begin{align*}
  L\left( \frac{1}{n}\sum_{k=1}^n x_k \right)
     \leq \frac{1}{n}\sum_{k=1}^n L(x_k)
          - \frac{\mu}{n} \sum_{k=1}^n \left(x_k - \frac{1}{n}\sum_{j=1}^n x_j \right)^2.
 \end{align*}
\end{Thm}

The notion of strong convexity with modulus $\mu\in\RR_+$ was generalized by Mak\'o et al.\ \cite{MakNikPal} as follows.

\begin{Def}\label{Def_strongly_convex_2}
Let $I\subseteq \RR$ be a nondegenerate interval, and $\alpha:(I-I)\to\RR_+$ be an even function.
A function $L:I\to\RR$ is said to be strongly convex with respect to the error function $\alpha$, briefly strongly $\alpha$-convex, if
 \begin{align}\label{help_strongly_alpha_convex_def}
  \begin{split}
  L(tx+(1-t)y)&\leq  tL(x)+(1-t)L(y)\\
              &\phantom{\leq\;}- t\alpha((1-t)(x-y))-(1-t)\alpha(t(y-x)),\qquad x,y\in I,\,t\in[0,1].
 \end{split}
 \end{align}
\end{Def}

Observe that if  $I\subseteq \RR$ is a nondegenerate interval and $\alpha:(I-I)\to\RR_+$ is defined by
$\alpha(u):=\mu u^2$, $u\in (I-I)$, where $\mu\in\RR_+$, then the notion of strong $\alpha$-convexity reduces to strong convexity with modulus $\mu\in\RR_+$. 
We note that, for certain error functions $\alpha$, the class of strongly $\alpha$-convex functions may be empty.
More precisely, if the necessary conditions in part (i) of the forthcoming Corollary 
 \ref{Cor_MakNikPal1} do not hold for $\alpha$, then there does not exist a strongly $\alpha$-convex function.
According to our knowledge, a satisfactory characterisation of those error functions $\alpha$ for which a strongly $\alpha$-convex function exists is not known.

In the next proposition, we collect some of the properties related to strong $\alpha$-convexity, that readily follow from Definition \ref{Def_strongly_convex_2}.

\begin{Pro}\label{Pro_basic_properties}
Let $I\subseteq \RR$ be a nondegenerate interval.
\begin{enumerate}[(i)]
 \item If $L_i:I\to\RR$ is strongly $\alpha_i$-convex for $i=1,2$, where $\alpha_1,\alpha_2:(I-I)\to\RR_+$ are even error functions
       and $c_1,c_2\in\RR_+$, then $c_1L_1+c_2L_2$ is strongly $(c_1\alpha_1+c_2\alpha_2)$-convex.
 \item If $\alpha_1,\alpha_2:(I-I)\to\RR_+$ are even error functions with $\alpha_1\leq\alpha_2$ and $L:I\to\RR$ is strongly $\alpha_2$-convex, then $L$ is strongly $\alpha_1$-convex as well.
 \item If, for each $n\in\NN$, $L_n:I\to\RR$ is strongly $\alpha_n$-convex, where $\alpha_n:(I-I)\to\RR_+$ is an even error function,
 and the sequences $(L_n)_{n\in\NN}$ and $ (\alpha_n)_{n\in\NN}$ pointwise converge to $L:I\to\RR$ and to $\alpha:(I-I)\to\RR_+$ as $n\to\infty$, respectively, then $L$ is strongly $\alpha$-convex.
\end{enumerate}
\end{Pro}

Next, we formulate a consequence of Theorem 8 in Mak\'o et al.\ \cite{MakNikPal}.

\begin{Thm}\label{Thm_MakNikPal}
Let $I\subseteq \RR$ be a nonempty open interval, $\alpha:(I-I)\to\RR_+$ be an even error function, and $L:I\to \RR$ be a function.
Then the following two assertions are equivalent:
 \begin{enumerate}
   \item[(i)] $L$ is strongly $\alpha$-convex, i.e., \eqref{help_strongly_alpha_convex_def} holds for all  $x,y\in I$ and $t\in[0,1]$.
   \item[(ii)] For all $x_0\in I$, there exists a constant $c_{x_0}\in\RR$ such that
         \begin{align}\label{help19}
           L(x)\geq L(x_0) + c_{x_0}(x-x_0) + \alpha(x-x_0),\qquad x\in I.
         \end{align}
 \end{enumerate}
\end{Thm}

As a consequence of Theorem \ref{Thm_MakNikPal}, we have that a strongly $\alpha$-convex function
 admits a so-called support function, which, roughly speaking,
 can be written as the sum of an affine function and a translate of the function $\alpha$.

We also specialize Corollaries 9 and 10 in Mak\'o et al.\ \cite{MakNikPal} to our setting.

\begin{Cor}\label{Cor_MakNikPal1}
Let $I\subseteq \RR$ be a nonempty open interval, $\alpha:(I-I)\to\RR_+$ be an even error function,
 and $L:I\to \RR$ be a strongly $\alpha$-convex function.
Then 
 \begin{itemize}
    \item[(i)] $\alpha(0)=0$ and
     \[
        \lim_{t\to 0} \frac{\alpha(t)}{t} =0,
    \]
    \item[(ii)] $L$ is a strongly $\widehat\alpha$-convex function as well, 
       where $\widehat\alpha:(I-I)\to\RR_+$ is defined by
       \begin{align}\label{def_alphahat}
        \widehat\alpha(u):=\sup\left\{  \frac{\alpha(tu)}{t} : t\in(0,1] \right\}, \qquad u\in (I-I).
      \end{align}
  \end{itemize}    
\end{Cor}

In the subsequent remark, we collect some properties of the function $\widehat\alpha$ defined in \eqref{def_alphahat}.

\begin{Rem}\label{Rem_starlikeness}
(i).
In view of part (i) of Corollary \ref{Cor_MakNikPal1}, 
 we can see that the conditions formulated therein are necessary for the nonemptiness of the class of strongly $\alpha$-convex functions.
Concerning part (ii) of Corollary \ref{Cor_MakNikPal1}, note that, by the definition of $\widehat\alpha$, we have that
$\widehat\alpha\geq\alpha$ is valid on $(I-I)$, and therefore, strong $\widehat\alpha$-convexity is a strengthening of strong $\alpha$-convexity. 

(ii).
Let $I\subseteq \RR$ be a nonempty open interval, $\alpha:(I-I)\to\RR_+$ be an even function with $\alpha(0)=0$.
Let us consider the function $\widehat\alpha$ defined in \eqref{def_alphahat}.
If $u\in I-I$ and $v\in(0,1]$, then $vu\in I-I$ (since $I-I$ is an interval and $0,u\in I-I$ yielding that $vu=vu+(1-v)0\in I-I$) and it holds that
 \Eq{uv}{
   \widehat\alpha(vu)
   &=\sup\left\{  \frac{\alpha(tvu)}{t} : t\in(0,1] \right\}
   =v\sup\left\{  \frac{\alpha(tvu)}{tv} : t\in(0,1] \right\}\\
   &=v\sup\left\{  \frac{\alpha(su)}{s} : s\in(0,v] \right\}
   \leq v\sup\left\{  \frac{\alpha(su)}{s} : s\in(0,1] \right\}
   =v\widehat\alpha(u).
   }
   
Observe that the inequality \eqref{uv} is equivalent to the increasingness of the map 
\Eq{au/u}{
(I-I)\cap(0,\infty)\ni u\mapsto\frac{\widehat\alpha(u)}{u}.
}
Indeed, if $0<w<u$, where $u\in (I-I)$, then $w\in (I-I)$ (since $I-I$ is an interval and $0\in I-I$) and, with  $v:=w/u$, the inequality \eq{uv} yields that
$\widehat\alpha(w)\leq \frac{w}{u} \widehat\alpha(u)$, which shows the increasingness of the map \eq{au/u}.
Conversely, assume that the map \eq{au/u} is increasing and let $u\in I-I$ and $v\in(0,1]$. 
If $u>0$, then $0<vu\leq u$ and $vu\in (I-I)$,
 whence $\frac{\widehat\alpha(vu)}{vu}\leq \frac{\widehat\alpha(u)}{u}$ follows. 
This implies that \eq{uv} is valid in the case $u\in I-I$, $u>0$, and $v\in(0,1]$.
The case when $u=0$ is trivial because $\widehat\alpha(0)=0$  (following from $\alpha(0)=0$), and the case when $u<0$ follows using that $\widehat\alpha$ is even (due to that $\alpha$ is even).

Taking into account that $\widehat\alpha(0)=0$, \eq{uv} implies that
 \begin{align*}
    \widehat\alpha(vu + (1-v)0)
        \leq v\widehat\alpha(u) + (1-v)\widehat\alpha(0), \qquad u\in I-I,   \qquad v\in[0,1],
 \end{align*}
i.e., $\widehat\alpha$ is a so-called starlike function with respect to the point $0\in (I-I)$ according to the terminology of P\'ales \cite{Pal2018}. 
For historical fidelity, we mention that the notion of starlikeness with respect to a given point is a particular version of conditional convexity introduced and studied by Losonczi \cite{Los2000}.

If, in addition, $\alpha$ is starlike with respect to $0$, i.e., 
 \[
    \alpha(vu + (1-v)0)
        \leq v\alpha(u) + (1-v)\alpha(0), \qquad u\in I-I,   \qquad v\in[0,1],
 \]
 or, equivalently, $\alpha(vu)\leq v \alpha(u)$, $u\in I-I$, $v\in[0,1]$, holds (following from $\alpha(0)=0$),
 then $\widehat \alpha=\alpha$.
Indeed, on the one hand, we have $\widehat\alpha\geq \alpha$ (see part (i)), and, on the other hand, 
 \[
  \widehat\alpha(u)
  = \sup\left\{  \frac{\alpha(tu)}{t} : t\in(0,1] \right\}
  \leq \sup\left\{  \frac{t\alpha(u)}{t} : t\in(0,1] \right\} =\alpha(u), \qquad u\in (I-I).
 \]
Consequently, for any even error function $\alpha:(I-I)\to\RR_+$ with $\alpha(0)=0$,
 using that $\widehat\alpha$ is starlike with respect to $0$ (checked above),
 we can conclude that $\widehat{\widehat\alpha}=\widehat\alpha$.
\proofend 
\end{Rem}

Next, we establish a particular case of part (ii) of Theorem 8 in Mak\'o et al.\ \cite{MakNikPal}, which offers a characterisation of differentiable strongly convex functions with respect to a nonnegative even error function. For reader's convenience, we provide a direct and simple proof.

\begin{Pro}\label{Prop_diff_strongly_convex}
Let $I\subseteq \RR$ be a nonempty open interval, $\alpha:(I-I)\to\RR_+$ be an even function, and $L:I\to \RR$ be a differentiable function.
Then $L$ is strongly $\alpha$-convex if and only if 
 \begin{align}\label{help18}
  L(x)\geq L(y) + L'(y)(x-y) + \alpha(x-y),\qquad x,y\in I.
 \end{align}
\end{Pro}

\begin{proof}
First, assume that $L$ is strongly $\alpha$-convex.
Then, by Theorem \ref{Thm_MakNikPal}, for all $x_0\in I$, there exists a constant $c_{x_0}\in\RR$ such that \eqref{help19} holds.
Hence, if $x>x_0$, $x,x_0\in I$, then
 \[
   \frac{L(x)-L(x_0)}{x-x_0}\geq c_{x_0}+\frac{\alpha(x-x_0)}{x-x_0}.
 \]
By taking the limit as $x\downarrow x_0$, using part (i) of Corollary \ref{Cor_MakNikPal1} and the differentiability of $L$,
 we get that $L'(x_0)\geq c_{x_0}$, $x_0\in I$.
With the same argument for $x<x_0$, $x,x_0\in I$, we have that $L'(x_0)\leq c_{x_0}$, $x_0\in I$. 
Therefore, $L'(x_0) = c_{x_0}$, $x_0\in I$, and, hence \eqref{help19} yields \eqref{help18}.

Now, we assume that \eqref{help18} holds.
Then the inequality in part (ii) of Theorem \ref{Thm_MakNikPal} holds with $x_0:=y$ and $c_{x_0}:=L'(y)$, where $y\in I$.
Therefore, Theorem \ref{Thm_MakNikPal} yields that $L$ is strongly $\alpha$-convex, as desired.
\end{proof}

\begin{Rem}
Let $I\subseteq \RR$ be a nonempty open interval, $\alpha:(I-I)\to\RR_+$ be an even function, and $L:I\to \RR$ be a differentiable function.
If $L$ is strongly $\alpha$-convex, then, by Proposition \ref{Prop_diff_strongly_convex} (by interchanging the roles of $x$ and $y$
 in \eqref{help18} and taking into account that $\alpha$ is even), we have that
 \begin{align}\label{help20}
    L(y)\geq L(x) + L'(x)(y-x) + \alpha(x-y),\qquad x,y\in I.
 \end{align}
By adding the inequalities \eqref{help18} and \eqref{help20} side by side, we get the following necessary condition for $L$ to be strongly $\alpha$-convex:
   \begin{align}\label{help21}
    (L'(x)-L'(y))(x-y)\geq 2\alpha(x-y),\qquad x,y\in I.
  \end{align}
However, we do not know whether, in general, the condition \eqref{help21} is sufficient or not.
In the special case $\alpha=0$, it is sufficient as well.
Indeed, in case of $\alpha=0$, the inequality \eqref{help21} takes the form 
 $(L'(x) - L'(y))(x-y)\geq 0$, $x,y\in I$, which yields that $L'$ is an increasing function on $I$. 
Therefore, using also the Lagrange mean value theorem, for any $x,y\in I$, $x\ne y$, there exists a point $\xi$
 between $x$ and $y$ such that
 \[
   L(x) - L(y) = L'(\xi)(x-y) \geq L'(y)(x-y),
 \] 
 that is, \eqref{help18} holds, which, by Proposition \ref{Prop_diff_strongly_convex}, implies that $L$ is strongly $\alpha$-convex.
\proofend  
\end{Rem}

Next, we present two nontrivial examples for strongly $\alpha$-convex functions.

\begin{Ex}\label{Ex_power_strongly_convex}
Let $I:=(0,\infty)$, $p\in[2,\infty)$, let $L:I\to\RR$ be defined by $L(x):=x^{p}$, $x\in I$, and $\alpha:\RR\to\RR$ be defined by $\alpha(t):=|t|^{p}$, $t\in\RR$. 
We get that $\alpha$ is continuous, even and strictly convex, and so it could serve as an example for part (iii) of our forthcoming Theorem \ref{Thm_monotonicity_alpha}.
Using Proposition \ref{Prop_diff_strongly_convex}, we are going to check that $L$ is strongly $\alpha$-convex.
Then $(I-I)=\RR$, $L'(x) = p x^{p-1}$, $x\in I$, and we verify that the inequality \eqref{help18} holds, i.e., 
 \[
  x^{p} - y^{p} \geq p y^{p-1}(x-y) + |x-y|^{p},\qquad x,y\in I,
 \]
 holds.
This inequality is equivalent to 
 \begin{align}\label{help22}
  x^{p} + (p-1)y^{p} - pxy^{p-1} - |x-y|^{p}\geq 0, \qquad x,y\in I.
 \end{align}
By dividing both sides of \eqref{help22} by $x^{p}$, 
 we have that \eqref{help22} is equivalent to 
 \[
 1+(p-1)\left(\frac{y}{x}\right)^{p}  - p\left(\frac{y}{x}\right)^{p-1} - \left|1-\frac{y}{x}\right|^{p}\geq 0,
   \qquad x,y\in I.
 \]
With the notation $t:=\frac{y}{x}\in I$, we have that \eqref{help22} is equivalent to 
 \begin{align}\label{help23}
  f(t):= 1 + (p-1)t^{p} - p t^{p-1} - |1-t|^{p}\geq 0, \qquad t\in I.
 \end{align}
If $p=2$, then \eqref{help23} holds trivially, since then $f(t)=1+t^2 -2t - |1-t|^2 = 0$, $t\in I$.
In what follows, assume that $p>2$. 
Then the derivative of $|1-t|^p$ with respect to $t$ at the point $1$ is $0$, since 
 \[
  \lim_{t\to 1} \frac{|1-t|^p}{t-1} =  \lim_{t\to 1} \Big(\sign(t-1)|1-t|^{p-1}\Big) =0,
 \]
 and therefore
 \begin{align*}
  f'(t)& = (p-1)p t^{p-1} -  (p-1)p t^{p-2} + p|1-t|^{p-2}(1-t)\\
       & = (p-1)p t^{p-2}(t-1) + p|1-t|^{p-2}(1-t) \\
       & = p(1-t) \big( |1-t|^{p-2} - (p-1)t^{p-2} \big)\\
       & = p(1-t) t^{p-2} \left(  \left| \frac{1}{t} - 1 \right|^{p-2} - p + 1 \right),
        \qquad t\in I.
 \end{align*}
Here $\left| \frac{1}{t} - 1 \right|^{p-2} - p + 1=0$ holds for some $t\in I$ if and only if 
 \[
  \left\vert \frac{1}{t} - 1 \right\vert = (p-1)^{\frac{1}{p-2}},
   \qquad \text{i.e.,}\qquad 
   t = \left(1\pm (p-1)^{\frac{1}{p-2}} \right)^{-1}. 
 \]
Since $p-1>1$, we have $(p-1)^{\frac{1}{p-2}}> 1$, yielding that 
 \[
     \left(1 -  (p-1)^{\frac{1}{p-2}} \right)^{-1} <0 \qquad \text{and}\qquad 
      \left(1 +  (p-1)^{\frac{1}{p-2}} \right)^{-1}\in (0,1).
 \]
Note also that $f'$ changes sign at $t=1$ from negative to positive, and at $t=\left(1 +  (p-1)^{\frac{1}{p-2}} \right)^{-1}$
 from positive to negative. 
This implies that the function $f$ has a local minimum at $t=1$, and a local maximum at 
 $t=\left(1 +  (p-1)^{\frac{1}{p-2}} \right)^{-1}$.
Since 
 \begin{align*}
    &\lim_{t\downarrow 0} f(t) = 1-1=0, \qquad f(1) = 1 + (p-1) -p=0,\\
    &\lim_{t\to \infty} f(t) = \lim_{t\to \infty} t^{p}\left( \frac{1}{t^{p}} + p-1 - \frac{p}{t} - \left|\frac{1}{t}-1\right|^{p}\right)
                             = \infty\cdot (p-1-1)
                             = \infty,  
 \end{align*}
 we have that $f(t)\geq f(1)=0$, $t\in I$, yielding that \eqref{help23} holds.
 
In this setting, the function $\widehat\alpha:\RR\to\RR_+$ defined in \eqref{def_alphahat} coincides with $\alpha$, since
 \Eq{*}{
  \widehat\alpha(u)
  &= \sup\left\{  \frac{\alpha(tu)}{t} : t\in(0,1] \right\}\\
  &= \sup\left\{  t^{p-1}|u|^{p} : t\in(0,1] \right\}
  = |u|^{p}
  =\alpha(u),\qquad u\in(I-I)=\RR.
 }
\proofend
\end{Ex}

\begin{Ex}\label{Ex_power_strongly_convex_2}
Let $I:=(0,\infty)$, let $L:I\to\RR$ be defined by $L(x):=\ee^{x}-x-1$, $x\in I$, and $\alpha:\RR\to\RR$ be defined by $\alpha(t):=\ee^{|t|}-|t|-1$, $t\in\RR$.
Using Example \ref{Ex_power_strongly_convex}, we check that $L$ is strongly $\alpha$-convex. 
Using part (i) of Proposition \ref{Pro_basic_properties} and Example \ref{Ex_power_strongly_convex} as well, 
 for each $n\geq 2$, $n\in\NN$, we get that the function 
 \ $(0,\infty)\ni x\mapsto \sum_{k=2}^n \frac{x^k}{k!}$ \ is strongly $\alpha_n$-convex, where $\alpha_n:\RR\to\RR$, $\alpha_n(t):=\sum_{k=2}^n \frac{\vert t\vert^k}{k!}$, $t\in\RR$.
Upon taking the limit $n\to\infty$ and using part (iii) of Proposition \ref{Pro_basic_properties}, the statement follows.
\proofend
\end{Ex}

In the next example, we point out that the function $(0,\infty)\ni x\mapsto x^2$ is strongly convex with respect to some error functions different from $\RR \ni x \mapsto x^2$, as well.

\begin{Ex}\label{Rem_triv_pelda_strongly_convex}
Using that $0\leq t\sin(t)\leq t^2$, $t\in(-\pi,\pi)$, and that $L:(0,\pi)\to(0,\infty)$, $L(x):=x^2$, $x\in(0,\pi)$, is strongly $\alpha_2$-convex with $\alpha_2(t):=t^2$, $t\in(-\pi,\pi)$ (following from Example \ref{Ex_power_strongly_convex} with $p=2$),
part (ii) of Proposition \ref{Pro_basic_properties} yields that $L$ is strongly $\alpha_1$-convex with $\alpha_1(t):=t\sin(t)$, $t\in(-\pi,\pi)$, as well.
In this case, for all $u\in(-\pi,\pi)$, we have
\[
 \widehat\alpha_1(u)
 =\sup\left\{  \frac{tu\sin(tu)}{t} : t\in(0,1] \right\}
 = \sup\left\{ u\sin(tu) : t\in(0,1] \right\}
 =\begin{cases}
   u\sin(u) & \mbox{if } |u|\leq \frac{\pi}{2},\\
   |u| & \mbox{if } \frac{\pi}{2}<|u|<\pi.
  \end{cases}
\]
Similarly, since $0\leq t\arctan(t)\leq t^2$, $t\in\RR$, we have that $L(x)=x^2$, $x\in\RR$ is strongly $\alpha_3$-convex
with $\alpha_3(t):=t\arctan(t)$, $t\in\RR$, as well. In this case, for all $u\in\RR$, we have
\Eq{*}{
 \widehat\alpha_3(u)
 &=\sup\left\{  \frac{tu\arctan(tu)}{t} : t\in(0,1] \right\}\\
 &= \sup\left\{ u\arctan(tu) : t\in(0,1] \right\}
 =u\arctan(u)
 =\alpha_3(u).
}
\proofend
\end{Ex}

Finally, we formulate a consequence of Theorem \ref{Thm_MakNikPal}, which is a generalization of Theorem \ref{Thm_strongly_convex}
 from strongly convex functions with some nonnegative modulus to strongly $\alpha$-convex functions,
 and from equal weights to convex weights.

\begin{Cor}\label{Cor_strong_alpha_convex}
Let $I\subseteq \RR$ be a nonempty open interval, $\alpha:(I-I)\to\RR_+$ be an even function,
 and $L:I\to \RR$ be a strongly $\alpha$-convex function.
Then
 \begin{align*}
   L\bigg(\sum_{i=1}^n t_ix_i\bigg)
   \leq \sum_{i=1}^n t_iL(x_i)
        -\sum_{i=1}^n t_i\alpha\bigg(x_i-\sum_{j=1}^n t_jx_j\bigg)
 \end{align*}
 for all $n\in\NN$, $x_1,\ldots,x_n\in I$ and $t_1,\ldots,t_n\in[0,1]$ satisfying $t_1+\cdots +t_n=1$.
\end{Cor}

\begin{proof}
Let $n\in\NN$, $x_1,\ldots,x_n\in I$ and $t_1,\ldots,t_n\in[0,1]$ satisfying $t_1+\cdots +t_n=1$ be arbitrarily fixed.
Using part (ii) of Theorem \ref{Thm_MakNikPal} with the choices $x:=x_i$, $i=1,\ldots,n$, and $x_0:=\sum_{i=1}^n t_ix_i$, we have that
 \begin{align*}
  L(x_i)\geq L(x_0) + c_{x_0}(x_i-x_0) + \alpha(x_i-x_0),\qquad i=1,\ldots,n,
 \end{align*}
where $c_{x_0}\in\RR$.
Multiplying this inequality by $t_i$, $i=1,\ldots,n$, and summing up the corresponding inequalities side by side, we can obtain that
 \begin{align*}
   \sum_{i=1}^n t_iL(x_i)
     \geq L(x_0) + c_{x_0}\sum_{i=1}^n t_i(x_i-x_0)
        + \sum_{i=1}^n t_i \alpha(x_i-x_0),
 \end{align*}
 where, by the definition of $x_0$,
 \[
    \sum_{i=1}^n t_i(x_i-x_0)
     = \sum_{i=1}^n t_i x_i- \left(\sum_{i=1}^nt_i\right)x_0
     = x_0-x_0
     = 0.
 \]
Hence
 \begin{align*}
   \sum_{i=1}^n t_iL(x_i)
      &\geq L(x_0) + \sum_{i=1}^n t_i \alpha(x_i-x_0)
      = L\bigg(\sum_{i=1}^n t_ix_i\bigg)
        +\sum_{i=1}^n t_i\alpha\bigg(x_i-\sum_{j=1}^n t_jx_j\bigg),
 \end{align*}
 which yields the statement.
\end{proof}

\section{Preliminaries on exchangeable random variables}\label{Sec_Mat_Gar}

First, we recall the notion of exchangeability of random variables.

\begin{Def}\label{Def_Exch}
Let $(\xi_\ell)_{\ell\in\NN}$ be a sequence of random variables.
Given $n\in \NN$, the random variables $\xi_1,\ldots,\xi_n$ are called {\sl exchangeable} if for any permutation $\sigma$ of $(1,2,\ldots,n)$,
 the law of $(\xi_{\sigma(1)},\ldots,\xi_{\sigma(n)})$ coincides with that of $(\xi_1,\ldots,\xi_n)$.
The sequence $(\xi_\ell)_{\ell\in\NN}$ is called exchangeable if, for each $n\in\NN$,
 the random variables $\xi_1,\ldots,\xi_n$ are exchangeable.
\end{Def}

If $(\xi_n)_{n\in\NN}$ is a sequence of independent and identically distributed random variables,
 then $(\xi_n)_{n\in\NN}$ is exchangeable as well.
Note that exchangeability is a weaker assumption than that of being independent and identically distributed,
but stronger than that of being identically distributed.

In the next proposition, we collect some basic properties of exchangeability that will be used later on, and, for completeness, we give a proof as well.

\begin{Pro}\label{Pro1}
The following assertions hold.
\begin{itemize}
 \item[(i)] If $n\in\NN\setminus\{1\}$ and $\xi_1,\ldots,\xi_n$ are exchangeable random variables, then the $(n-1)$-dimensional random vectors
 \[
   (\xi_2,\xi_3,\ldots,\xi_n), \quad (\xi_1,\xi_3,\ldots,\xi_n),\quad \ldots\quad (\xi_1,\xi_2,\ldots,\xi_{n-1})
 \]
 have the same distributions.
 
\item[(ii)]
If $h:\RR\to\RR$ is a Borel measurable function and $(\xi_n)_{n\in\NN}$ is a sequence of exchangeable random variables, then $(h(\xi_n))_{n\in\NN}$ is also exchangeable.

\item[(iii)]
If $n\in\NN$, $\lambda_1,\ldots,\lambda_n\in\RR$, and $\xi_1,\ldots,\xi_n$ are exchangeable random variables,
 then, for each permutation $\sigma$ of $(1,\ldots,n)$, we have $\sum_{i=1}^n \lambda_{\sigma(i)} \xi_i$ has the same distribution as $\sum_{i=1}^n \lambda_i \xi_i$.
\end{itemize}
\end{Pro}
 
\begin{proof}
(i): In order to check that $(\xi_2,\xi_3,\ldots,\xi_n)$ and $(\xi_1,\xi_3,\ldots,\xi_n)$ have the same distributions,
 let us consider the permutation $\sigma(1):=2$, $\sigma(2):=1$, $\sigma(i):=i$, $i\in\{3,4,\ldots,n\}$ of $(1,\ldots,n)$.
Then, by the exchangeability of $\xi_1,\ldots,\xi_n$, we have that the $n$-dimensional random vectors
$(\xi_1,\xi_2,\xi_3,\ldots,\xi_n)$ and $(\xi_2,\xi_1,\xi_3,\ldots,\xi_n)$ have the same distributions, yielding that the $(n-1)$-dimensional random vectors consisting of their last $(n-1)$ coordinates have the same distributions, i.e., $(\xi_2,\xi_3,\ldots,\xi_n)$ and $(\xi_1,\xi_3,\ldots,\xi_n)$ have the same distributions, as desired.
The other equality in distributions can be checked similarly by choosing an appropriate permutation of $(1,\ldots,n)$.

(ii): For each $n\in\NN$, let us consider the function $H_n:\RR^n\to\RR^n$, $H_n(x_1,\ldots,x_n):=(h(x_1),\ldots,h(x_n))$, $x_1,\ldots,x_n\in\RR$.
Then $H_n$ is Borel measurable, since the mapping $\RR^n\ni(x_1,\ldots,x_n)\mapsto h(x_i)$ is Borel measurable for each $i\in\{1,\ldots,n\}$ (following from the fact that $\cB(\RR^n)$ coincides with the $n$-fold product of the $\sigma$-algebra $\cB(\RR)$).
Further, for each $n\in\NN$, by the exchangeability of $\xi_1,\ldots,\xi_n$,
 for any permutation $\sigma$ of $(1,2,\ldots,n)$,
 the law of $(\xi_{\sigma(1)},\ldots,\xi_{\sigma(n)})$ coincides with that of $(\xi_1,\ldots,\xi_n)$.
Since $H_n$ is Borel measurable, it implies that, for each $n\in\NN$ and any permutation $\sigma$
 of $(1,2,\ldots,n)$, the law of $H_n(\xi_{\sigma(1)},\ldots,\xi_{\sigma(n)})$ coincides with that of $H_n(\xi_1,\ldots,\xi_n)$.
This yields that, for any permutation $\sigma$ of $(1,2,\ldots,n)$, the law of $(h(\xi_{\sigma(1)}),\ldots,h(\xi_{\sigma(n)}))$ coincides with that of $(h(\xi_1),\ldots,h(\xi_n))$, as desired.

(iii): For any permutation $\sigma$ of $(1,2,\ldots,n)$, we have that
 \begin{align}\label{help10}
   \sum_{i=1}^n \lambda_{\sigma(i)} \xi_i = \sum_{i=1}^n \lambda_i \xi_{\sigma^{-1}(i)},
 \end{align}
 where $\sigma^{-1}$ denotes the inverse permutation of $\sigma$.
By the exchangeability of $\xi_1,\ldots,\xi_n$, we have that $(\xi_1,\ldots,\xi_n)$ and $(\xi_{\sigma^{-1}(1)},\ldots,\xi_{\sigma^{-1}(n)})$
 have the same distributions, which implies that $\sum_{i=1}^n \lambda_i \xi_{\sigma^{-1}(i)}$
 and $\sum_{i=1}^n \lambda_i \xi_i$ have the same distributions (following from the fact that the mapping $\RR^n\ni(x_1,\ldots,x_n)\mapsto \sum_{i=1}^n \lambda_i x_i$ is Borel measurable).
As a consequence, using \eqref{help10}, we get that $\sum_{i=1}^n \lambda_{\sigma(i)} \xi_i$ and $\sum_{i=1}^n \lambda_i \xi_i$ have the same distributions, as desired.
\end{proof} 
 
In statistics, forecasting based on exchangeable random variables, 
 in particular, on independent and identically distributed random variables, is of high importance. 
Motivated by an example due to Mattei and Garreau \cite[Section 3.3]{MatGar}, in which they pointed out that the convex loss of an ensemble forecasting can get worse in average when adding a new component to it, we present an example in which we demonstrate that the convex loss of a forecasting based on exchangeable random variables with non-equal weights may get worse in average compared to the one with equal weights.
In fact, our example is also a generalization of the above mentioned example of Mattei and Garreau \cite[Section 3.3]{MatGar}.

\begin{Ex}\label{Ex_MatGar}
Let $n\in\NN$, and $\lambda_1,\dots,\lambda_n\in[0,1]$ with $\lambda_1+\dots+\lambda_n=1$.
Further, let $\xi_1,\dots,\xi_n$ be exchangeable random variables.
For $i\in\{1,\dots,n\}$ and $j\in\{0,\dots,n-1\}$, define the truncated sum $i\oplus j$ of $i$ and $j$ by
       \begin{align*}
         i\oplus j
         :=\begin{cases}
          i+j &\mbox{if } i+j\leq n,\\
          i+j-n &\mbox{if } i+j>n.
          \end{cases}
       \end{align*}
It is clear that $(i,i\oplus 1,\dots,i\oplus (n-1))$, $i\in\{1,\ldots,n\}$, are different cyclic permutations of $(1,\dots,n)$.
Hence, taking into that $\lambda_1+\dots+\lambda_n=1$, one can easily see that the identity
       \begin{align}\label{help_Mat_example_1}
        \frac{\xi_1+\dots+\xi_n}{n}
        = \frac{1}{n} \sum_{j=0}^{n-1}
        \big(\lambda_{1\oplus j}\xi_1+\dots+\lambda_{n\oplus j}\xi_n\big)
       \end{align}
holds. In what follows, let  $L:\RR\to\RR$ be a convex function, and suppose that 
  \begin{align}\label{help17}
  \EE\big( \vert L(\lambda_{1}\xi_1+\dots+\lambda_{n}\xi_n) \vert \big)<\infty.
  \end{align}
Here $L(\lambda_{1}\xi_1+\dots+\lambda_{n}\xi_n)$ is indeed a random variable, since a convex function
 defined on $\RR$ is continuous and hence it is Borel measurable as well.
Then \eqref{help_Mat_example_1} implies that 
       \begin{align}\label{help_Mat_example_2}
         L\Big(\frac{\xi_1+\dots+\xi_n}{n} \Big)
         \leq \frac{1}{n} \sum_{j=0}^{n-1}
        L(\lambda_{1\oplus j}\xi_1+\dots+\lambda_{n\oplus j}\xi_n).
       \end{align}
Since $\xi_1,\dots,\xi_n$ are exchangeable and $(1\oplus j,\dots,n\oplus j)$ is a permutation of $(1,\dots,n)$ for each $j\in\{0,1,\ldots,n-1\}$,
 by part (iii) of Proposition \ref{Pro1}, we have that the random variables
 $\lambda_{1\oplus j}\xi_1+\dots+\lambda_{n\oplus j}\xi_n$, $j\in\{0,\dots,n-1\}$,  
 have the same law, which, by the assumption \eqref{help17}, in particular, yields that
 \[
  \EE(\vert L(\lambda_{1\oplus j}\xi_1+\dots+\lambda_{n\oplus j}\xi_n)\vert ) < \infty,
   \qquad j\in\{0,1,\ldots,n-1\}.
 \]
Therefore, by taking expectation of both sides of \eqref{help_Mat_example_2}, we have that
       \begin{align*}
         \EE\left( L\Big(\frac{\xi_1+\dots+\xi_n}{n} \Big)\right)
         \leq \frac{1}{n} \sum_{j=0}^{n-1}
        \EE\big( L(\lambda_{1\oplus j}\xi_1+\dots+\lambda_{n\oplus j}\xi_n)\big)
         = \EE\big( L(\lambda_{1}\xi_1+\dots+\lambda_{n}\xi_n)\big),
       \end{align*}
and $\EE\left( L\Big(\frac{\xi_1+\dots+\xi_n}{n} \Big)\right)\in[-\infty,\infty)$.   
By choosing $n=2$, $\lambda_1=\frac{2}{3}$ and $\lambda_2=\frac{1}{3}$, we get back the example of
Mattei and Garreau \cite[Section 3.3]{MatGar} under the weaker assumption that $\xi_1$ and $\xi_2$ are exchangeable instead of $\xi_1$ and $\xi_2$ are independent.
\proofend
\end{Ex}

\section{Results for convex loss functions and strongly convex loss functions with a nonnegative modulus}\label{Sec_convex_loss}

First, we recall a result on the arithmetic means of exchangeable random variables and convex loss functions due to Marshall and Proschan \cite{MarPro} (see also Marshall et al.\ \cite[Proposition B.2.b on page 395]{MarOlk}).
In Proposition B.2.b on page 395 of the book \cite{MarOlk} by Marshall et al., the above mentioned result was stated (with a missing moment condition) for independent and identically distributed random variables, but it was also noted (after the proof) that the statement was valid for exchangeable random variables as well. 
We mention that this result has also recently been recalled by Mattei and Garreau \cite[Theorem 2]{MatGar}.
 
\begin{Thm}\label{Thm_Mat_Gar_1}
Let $I\subseteq \RR$ be a nonempty open interval and $(\xi_n)_{n\in\NN}$ be a sequence of exchangeable random variables with values in $I$.
If $L:I\to\RR$ is a convex (loss) function such that $\EE(\vert L(\xi_1)\vert)<\infty$, then
 $\EE(L(\frac{\xi_1+\cdots+\xi_n}{n}))\in[-\infty,\infty)$, $n\in\NN$, and the sequence $\big(\EE(L(\frac{\xi_1+\cdots+\xi_n}{n}))\big)_{n\in\NN}$ is decreasing.
\end{Thm}

Next, we present a consequence of Theorem \ref{Thm_Mat_Gar_1} for quasi-arithmetic means

\begin{Thm}\label{Thm_monotonicity}
Let $I\subseteq \RR$ be a nonempty open interval, $f:I\to\RR$ be a continuous and strictly increasing function, and $(\xi_n)_{n\in\NN}$ be a sequence of exchangeable random variables with values in $I$.
If $L:I\to\RR$ is a function such that the function $L\circ f^{-1}$ is convex on $f(I)$
 and $\EE(\vert L(\xi_1)\vert)<\infty$, then $\EE\big(L\big(\mathscr{A}_n^{f}(\xi_1,\ldots,\xi_n)\big) \big)\in[-\infty,\infty)$, $n\in\NN$, and
 the sequence $\Big(\EE\big(L\big( \mathscr{A}_n^{f}(\xi_1,\ldots,\xi_n) \big)\big)\Big)_{n\in\NN}$ is decreasing.
\end{Thm}

\begin{proof}
We apply Theorem \ref{Thm_Mat_Gar_1} by replacing $I$ by $f(I)$, $(\xi_n)_{n\in\NN}$ by $(f(\xi_n))_{n\in\NN}$ and $L$ by $L\circ f^{-1}$, respectively.
We check that one can indeed apply Theorem \ref{Thm_Mat_Gar_1} with these choices.
Since $I$ is a nonempty open interval and $f$ is strictly increasing and continuous, we have that $f(I)$ is a nonempty open interval.
Since $f$ is continuous, $f(\xi_n)$, $n\in\NN$, are random variables as well, and the exchangeability of $(\xi_n)_{n\in\NN}$ yields that of $(f(\xi_n))_{n\in\NN}$ (see part (ii) of Proposition \ref{Pro1}).
Using that 
 \begin{align}\label{help26}
    (L\circ f^{-1})\left(\frac{f(\xi_1)+\cdots+f(\xi_n)}{n}\right)
            =  L(\mathscr{A}_n^{f}(\xi_1,\ldots,\xi_n)),\qquad n\in\NN,
 \end{align}
Theorem \ref{Thm_Mat_Gar_1} readily implies the statement.
\end{proof}

The following Theorem \ref{Thm_strict_monotonicity} is about a connection of quasi-arithmetic means of
 exchangeable random variables and strongly convex loss functions with a nonnegative modulus,
 and it is a generalization of Theorem 4 in dimension 1 of Mattei and Garreau \cite{MatGar} 
 from arithmetic means to quasi-arithmetic means.
  
\begin{Thm}\label{Thm_strict_monotonicity}
Let $I\subseteq \RR$ be a nonempty open interval, $f:I\to\RR$ be a continuous and strictly increasing function,
 and $(\xi_n)_{n\in\NN}$ be a sequence of exchangeable random variables with values in $I$ such that
 $\DD^2(f(\xi_1))\in(0,\infty)$.
Let $\mu\in\RR_+$ and $L:I\to\RR$ be a function such that the function $L\circ f^{-1}: f(I)\to\RR$ is strongly convex with modulus $\mu$
 and $\EE(\vert L(\xi_1)\vert)<\infty$.
Then, for each $n\in\NN$, we have that $\EE\big(L\big( \mathscr{A}_n^{f}(\xi_1,\ldots,\xi_n) \big)\big)\in[-\infty,\infty)$ and the sequence
\[
  \bigg(\EE\big(L\big( \mathscr{A}_n^{f}(\xi_1,\ldots,\xi_n) \big)\big)-\frac{\mu(1-\varrho)}{n}\DD^2(f(\xi_1))\bigg)_{n\in\NN}
\]
is decreasing, where
 \[
   \varrho:=\frac{\cov(f(\xi_1),f(\xi_2))}{\sqrt{ \DD^2(f(\xi_1)) \DD^2(f(\xi_2)) } }
           = \corr(f(\xi_1),f(\xi_2)).
 \]
\end{Thm}

\begin{proof}
We can apply Theorem 4 in Mattei and Garreau \cite{MatGar}, 
 where the statement was proved for arithmetic mean, 
 which is a quasi-arithmetic mean corresponding to the identity function.
Since $I$ is a nonempty open interval and $f$ is strictly increasing and continuous,
 we have that $f(I)$ is also a nonempty open interval.
Since $f$ is continuous, $f(\xi_n)$, $n\in\NN$, are random variables as well, 
 and the exchangeability of $(\xi_n)_{n\in\NN}$ yields that of $(f(\xi_n))_{n\in\NN}$ (see part (ii) of Proposition \ref{Pro1}).
By applying Theorem 4 in Mattei and Garreau \cite{MatGar} in dimension 1
 for the nonempty open interval $f(I)$, 
 the exchangeable sequence of random variables $(f(\xi_n))_{n\in\NN}$ and for $L\circ f^{-1}$ being a strongly convex function with modulus $\mu$, 
 with the help of \eqref{help26},  
  for each $n\in\NN\setminus\{1\}$, we obtain that $\EE\big(L\big( \mathscr{A}_n^{f}(\xi_1,\ldots,\xi_n) \big)\big)\in[-\infty,\infty)$ and
 \[
   \EE\big(L\big( \mathscr{A}_n^{f}(\xi_1,\ldots,\xi_n) \big)\big)
       \leq \EE\big(L\big( \mathscr{A}_{n-1}^{f}(\xi_1,\ldots,\xi_{n-1}) \big)\big)
             - \frac{\mu(1-\varrho) }{n(n-1)}\DD^2(f(\xi_1)).
 \]
 By rearranging this inequality, we have that 
 \[
 \EE\big(L\big( \mathscr{A}_n^{f}(\xi_1,\ldots,\xi_n) \big)\big)
     - \frac{\mu(1-\varrho)}{n}\DD^2(f(\xi_1))
       \leq \EE\big(L\big( \mathscr{A}_{n-1}^{f}(\xi_1,\ldots,\xi_{n-1}) \big)\big)
             - \frac{\mu(1-\varrho) }{n-1}\DD^2(f(\xi_1))
 \]
 for each $n\in\NN\setminus\{1\}$, yielding the statement.
\end{proof}

Theorem \ref{Thm_strict_monotonicity} has the following simple consequence.

\begin{Rem}\label{Rem_strongly_convex_mod_loss}
(i).
Under the conditions of Theorem \ref{Thm_strict_monotonicity},
 if, in addition, $\EE\big(L\big( \mathscr{A}_n^{f}(\xi_1,\ldots,\xi_n) \big)\big)>-\infty$, $n\in\NN\setminus\{1\}$, $\mu>0$,
  and $\varrho\ne 1$, then the sequence
 $\Big(\EE\big(L\big( \mathscr{A}_n^{f}(\xi_1,\ldots,\xi_n) \big)\big)\Big)_{n\in\NN}$ is strictly decreasing.
Indeed, for each $n\in\NN$, it holds that
 \begin{align*}
   \EE\big(L\big( \mathscr{A}_{n+1}^{f}(\xi_1,\ldots,\xi_{n+1}) \big)\big)
       - \frac{\mu(1-\varrho)}{n+1}\DD^2(f(\xi_1))
    \leq \EE\big(L\big( \mathscr{A}_n^{f}(\xi_1,\ldots,\xi_n) \big)\big)-\frac{\mu(1-\varrho)}{n}\DD^2(f(\xi_1)),
 \end{align*}
 yielding that 
\begin{align*}
\EE\big(L\big( \mathscr{A}_{n+1}^{f}(\xi_1,\ldots,\xi_{n+1}) \big)\big)
 - \EE\big(L\big( \mathscr{A}_n^{f}(\xi_1,\ldots,\xi_n) \big)\big)
  \leq - \mu(1-\varrho) \DD^2(f(\xi_1)) \frac{1}{n(n+1)}<0.
\end{align*}

(ii). The correction term $\frac{\mu(1-\varrho)}{n}\DD^2(f(\xi_1))$ in Theorem \ref{Thm_strict_monotonicity} 
 depends only on the joint distribution of $f(\xi_1)$ and $f(\xi_2)$.
\proofend 
\end{Rem}

Note that Theorem \ref{Thm_strict_monotonicity} with $\mu=0$ gives back
Theorem \ref{Thm_monotonicity} under the additional condition $\DD^2(f(\xi_1))\in(0,\infty)$.

In the next example, we specialize Theorem \ref{Thm_strict_monotonicity}
 to the case of geometric mean and a sequence of i.i.d.\
 random variables, and, in the second part of this example, we make a further specialization by choosing the uniform distribution
 on $(0,1)$ as the common distribution of the i.i.d.\ sequence of random variables in question.

\begin{Ex}\label{Ex_Sect5}
We apply Theorem \ref{Thm_strict_monotonicity} in the following setting: $I:=(0,\infty)$, $f:(0,\infty)\to\RR$, $f(x):=\ln(x)$, $x>0$, and $(\xi_n)_{n\in\NN}$ is a sequence of independent and identically distributed positive random variables
 such that $\EE((\ln(\xi_1))^2)<\infty$ and $\PP(\xi_1=\EE(\xi_1))<1$.
Then $\ln(I)=\RR$ and $\DD^2(\ln(\xi_1))\in(0,\infty)$.
Let $L:I\to\RR$, $L(x):=(\ln(x))^2$, $x\in I$.
Then $\EE(\vert L(\xi_1)\vert) = \EE((\ln(\xi_1))^2)<\infty$, and
 the function $\RR\ni x\mapsto (L\circ \ln^{-1})(x) = x^2$ is strongly convex with modulus $\mu=1$
 (see the paragraph after Definition \ref{Def_strongly_convex_3}), and   
 \[
   \cA_n^{\ln}(\xi_1,\ldots,\xi_n) = (\xi_1\cdots\xi_n)^{\frac{1}{n}},\qquad n\in\NN.
 \]
By Theorem \ref{Thm_strict_monotonicity} and taking into account that $\varrho=0$, we have that
 \[
    \EE\Big[ \big( \ln\big( (\xi_1\cdots\xi_n)^{\frac{1}{n}} \big) \big)^2 \Big]
     = \frac{1}{n^2} \EE\left(  \left(  \sum_{i=1}^n \ln(\xi_i) \right)^2  \right)
     <\infty, \qquad n\in\NN,
 \]
 and that the sequence  
 \begin{align}\label{help_strong_mod1_pelda_1}
 \begin{split}
     \left( \frac{1}{n^2} \EE\left(  \left(  \sum_{i=1}^n \ln(\xi_i) \right)^2  \right)  - \frac{1}{n} \DD^2(\ln(\xi_1)) \right)_{n\in\NN}
 \end{split}                    
 \end{align}
 is decreasing.
As a consequence, the sequence $\left( \frac{1}{n^2} \EE\left(  \left(  \sum_{i=1}^n \ln(\xi_i) \right)^2  \right)  \right)_{n\in\NN}$
 is strictly decreasing in accordance with Remark \ref{Rem_strongly_convex_mod_loss}.
 
Next, we specialize \eqref{help_strong_mod1_pelda_1} to the case when $\xi_1$ is uniformly distributed on the interval $(0,1)$. 
Then $-\ln(\xi_1)$ is exponentially distributed with parameter $1$, since 
 $\PP(-\ln(\xi_1) < x)=0$ for all $x\leq0$, and
 \begin{align*}
 \PP(-\ln(\xi_1) < x) = 1 - \PP(\ln(\xi_1) \leq -x)
                      = 1 - \PP(\xi_1 \leq \ee^{-x})
                      = 1 - \ee^{-x}, \qquad x>0.
 \end{align*}
It implies that $\sum_{i=1}^n (-\ln(\xi_i))$, $n\in\NN$, is Gamma distributed with parameters $n$ and $1$, and hence
 \[
    \EE\left(\left( \sum_{i=1}^n (-\ln(\xi_i)) \right)^2\right)
        = \frac{\Gamma(n+2)}{\Gamma(n)} 
        = n(n+1),\qquad n\in\NN,
 \]  
 where $\Gamma$ denotes the Gamma function.
Moreover,
 \begin{align}\label{help25}
  \EE( (\ln(\xi_1))^n ) = (-1)^n \frac{\Gamma(1+n)}{\Gamma(1)} = (-1)^n n!,\qquad n\in\NN,
 \end{align} 
 yielding that 
 \begin{align}\label{help_Var_Exp}
 \DD^2( \ln(\xi_1)) = \EE( (\ln(\xi_1))^2 ) - (\EE( \ln(\xi_1) ))^2 = (-1)^22-(-1)^2=1.
 \end{align}
Therefore, in the considered special case, \eqref{help_strong_mod1_pelda_1} means that the sequence 
 $\big(\frac{n+1}{n} - \frac{1}{n} \big)_{n\in\NN} = (1)_{n\in\NN}$ is decreasing, which readily holds.
Further, in the considered special case, one can also see that
there does not exist a constant $c>\DD^2( \ln(\xi_1))=1$ such that
 the sequence
\[
  \left( \frac{1}{n^2} 
   \EE\left(  \left(  \sum_{i=1}^n \ln(\xi_i) \right)^2  \right)  - \frac{c}{n} \right)_{n\in\NN}
\]
is decreasing, since, if $c>1$, the sequence $\frac{n+1}{n} - \frac{c}{n} = 1 + \frac{1-c}{n}$, $n\in\NN$,
 is strictly increasing.
\proofend
\end{Ex}

\section{Results for strongly convex loss functions with respect to a nonnegative even error function}\label{Sec_convex_loss_3}

The following theorem is the main result of the paper.

\begin{Thm}\label{Thm_monotonicity_alpha}
Let $I\subseteq \RR$ be a nonempty open interval, $f:I\to\RR$ be a continuous and strictly increasing function, and
 $\alpha:(f(I)-f(I))\to\RR_+$ be a Borel measurable even function.
Further, let $(\xi_n)_{n\in\NN}$ be a sequence of exchangeable random variables with values in $I$, 
 and $L:I\to\RR$ be a function such that the function $L\circ f^{-1}: f(I)\to\RR$ is strongly $\alpha$-convex and $\EE(\vert L(\xi_1)\vert)<\infty$.
Then the following assertions hold:
 \begin{itemize}
  \item[(i)] For each $n\in\NN\setminus\{1\}$, we have that $\EE\big(L\big( \mathscr{A}_n^{f}(\xi_1,\ldots,\xi_n) \big)\big)\in[-\infty,\infty)$ and
 \begin{align}\label{help6}
  \begin{split}
   \EE\big(L\big( \mathscr{A}_n^{f}(\xi_1,\ldots,\xi_n) \big)\big)
   &\leq \EE\big(L\big( \mathscr{A}_{n-1}^{f}(\xi_1,\ldots,\xi_{n-1}) \big)\big) \\
   &\phantom{\leq\;} - \EE\left( \alpha\left( \frac{1}{n(n-1)}\sum_{i=1}^{n-1} f(\xi_i) - \frac{1}{n}f(\xi_n) \right)\right).
  \end{split}
 \end{align}
 \item[(ii)]
The sequence $\Big(\EE\big(L\big( \mathscr{A}_n^{f}(\xi_1,\ldots,\xi_n) \big)\big)\Big)_{n\in\NN}$ is decreasing.
 \item[(iii)] If, in addition, 
            \begin{itemize}
              \item[(1)] $\DD^2(f(\xi_1))\in(0,\infty)$,
              \item[(2)] $\cov(f(\xi_1),f(\xi_2))\leq 0$,
              \item[(3)] $\alpha$ is strictly convex,
              \item[(4)] $\EE\left( \alpha\left( \frac{1}{2}(f(\xi_1) - f(\xi_2)) \right)\right)<\infty$,
              \item[(5)] $\EE\big(L\big( \mathscr{A}_n^{f}(\xi_1,\ldots,\xi_n) \big)\big)>-\infty$, $n\in\NN$,
            \end{itemize}
            then the sequence $\Big(\EE\big(L\big( \mathscr{A}_n^{f}(\xi_1,\ldots,\xi_n) \big)\big)\Big)_{n\in\NN}$ is strictly decreasing.
 \end{itemize}
\end{Thm}

\begin{proof}
(i).
In what follows, let $n\in\NN\setminus\{1\}$ be arbitrarily fixed, and $\Lambda:=L\circ f^{-1}$.
Since $I$ is a nonempty open interval and $f$ is strictly increasing and continuous,
 we have that $f(I)$ is also a nonempty open interval.
Hence $\Lambda$ is continuous and, in particular, is Borel measurable,
 since a strongly $\alpha$-convex function is convex (due to the fact that $\alpha$ is nonnegative),
 and a convex function defined on an open interval is continuous.
Since $f(I)$ is a nonempty open interval and $\Lambda$ is strongly $\alpha$-convex, by Corollary \ref{Cor_strong_alpha_convex},
 we have that the inequality
 \begin{align}\label{strongly_convex_ineq4}
  \Lambda\left( \frac{1}{n}\sum_{k=1}^n y_k \right)
     \leq \frac{1}{n}\sum_{k=1}^n \Lambda(y_k)
          - \frac{1}{n} \sum_{k=1}^n \alpha \left(y_k - \frac{1}{n}\sum_{j=1}^n y_j \right)
     \leq \frac{1}{n}\sum_{k=1}^n \Lambda(y_k)
 \end{align}
 holds for all $y_1,\ldots,y_n\in f(I)$.
By choosing
 \[
   y_k:=\frac{1}{n-1}\sum_{i\in\{1,\ldots,n\}\setminus\{k\} } f(\xi_i), \qquad k\in\{1,\ldots,n\},
 \]
 in the first inequality in \eqref{strongly_convex_ineq4}, we have that $y_1+\cdots+y_n = f(\xi_1)+\cdots+f(\xi_n)$ and
 \begin{align}\label{strongly_convex_ineq5}
  \begin{split}
  \Lambda\left( \frac{1}{n}\sum_{k=1}^n f(\xi_k) \right) 
  &\leq \frac{1}{n}\sum_{k=1}^n \Lambda\left( \frac{1}{n-1}\sum_{i\in\{1,\ldots,n\}\setminus\{k\}} f(\xi_i) \right)\\
  &\;\; - \frac{1}{n} \sum_{k=1}^n \alpha\left( \frac{1}{n-1}\sum_{i\in\{1,\ldots,n\}\setminus\{k\}} f(\xi_i) - \frac{1}{n}\sum_{j=1}^n f(\xi_j) \right).
 \end{split}
 \end{align}
By choosing $y_k:=f(\xi_k)$, $k\in\{1,\ldots,n\}$, in the second inequality in \eqref{strongly_convex_ineq4}, we obtain that
 \begin{align}\label{strongly_convex_ineq5_b}
  \Lambda\left( \frac{1}{n}\sum_{k=1}^n f(\xi_k) \right)
     \leq \frac{1}{n} \sum_{k=1}^n \Lambda(f(\xi_k)).
 \end{align} 
Since $\Lambda$ is Borel measurable, we have that $\Lambda(f(\xi_k))$, $k\in\{1,\ldots,n\}$, 
 and $\Lambda\left( \frac{1}{n}\sum_{k=1}^n f(\xi_k) \right)$, $n\in\NN$, are random variables.
Using that $\EE(\vert L(\xi_1)\vert)<\infty$ and that $\xi_1,\ldots,\xi_n$ have the same distributions (due to their exchangeability),
 by taking the expectation of both sides of the inequality \eqref{strongly_convex_ineq5_b}, we obtain
 \begin{align*}
   \EE\left(\Lambda\left(\frac{1}{n}\sum_{k=1}^n f(\xi_k) \right) \right) 
     \leq  \frac{1}{n} \sum_{k=1}^n \EE(\Lambda(f(\xi_k))) 
      = \frac{1}{n} \sum_{k=1}^n  \EE(L(\xi_k)) 
      = \EE(L(\xi_1)).
 \end{align*}
In particular, it implies that
 \begin{align}\label{help15}
 \EE\left(\Lambda\left( \frac{1}{n}\sum_{k=1}^n f(\xi_k) \right)\right) \in[-\infty,\infty).
 \end{align}
Similarly to \eqref{help15}, we also have that
 \begin{align}\label{help29}
   \EE\left( \Lambda\left( \frac{1}{n-1}\sum_{ i\in\{1,\ldots,n\}\setminus\{k\}} f(\xi_i) \right)\right)
      \in[-\infty,\infty), \qquad k\in\{1,\ldots,n\}.
 \end{align}
Since the sequence $(\xi_\ell)_{\ell\in\NN}$ is exchangeable and $f$ is Borel measurable,
 we have that the sequence $(f(\xi_\ell))_{\ell\in\NN}$ is exchangeable as well, see part (ii) of  Proposition \ref{Pro1}.
Together with part (iii) of Proposition \ref{Pro1} with the choices $\lambda_i:=\frac{1}{n-1}$, $i\in\{1,\ldots,n-1\}$,
 and $\lambda_n:=-1$, it implies that, for each $k\in\{1,\ldots,n\}$, the random variables
 \begin{align*}
   \frac{1}{n-1}\sum_{i\in\{1,\ldots,n\}\setminus\{k\}} f(\xi_i)-f(\xi_k)
  \qquad \text{and} \qquad \frac{1}{n-1}\sum_{i=1}^{n-1}f(\xi_i)-f(\xi_n)
 \end{align*}
have the same distributions.
Hence, using also that $\alpha$ is nonnegative, we get that
 \begin{align*}
  &\EE\left( \alpha\left( \frac{1}{n-1}\sum_{ i\in\{1,\ldots,n\}\setminus\{k\}} f(\xi_i)
                      - \frac{1}{n}\sum_{j=1}^n f(\xi_j) \right)\right) \\
  &\qquad = \EE\left( \alpha\left( \frac{1}{n(n-1)}\sum_{i\in\{1,\ldots,n\}\setminus\{k\}} f(\xi_i)
                             - \frac{1}{n}f(\xi_k) \right) \right)\\
  &\qquad = \EE\left( \alpha\left( \frac{1}{n(n-1)}\sum_{i=1}^{n-1} f(\xi_i) - \frac{1}{n}f(\xi_n) \right)\right)\in[0,\infty],
  \qquad k\in\{1,\ldots,n\}.
 \end{align*}
Therefore, using also \eqref{help29}, the expected value of the right hand side of the inequality \eqref{strongly_convex_ineq5} exists and it belongs to $[-\infty,\infty)$.
Hence, by taking the expectation of both sides of the inequality \eqref{strongly_convex_ineq5}, we get that
 \begin{align}\label{help5}
  \begin{split}
  &\EE\left(\Lambda\left( \frac{1}{n}\sum_{k=1}^n f(\xi_k) \right)\right)  \\
  &\qquad \leq \EE\left(\Lambda\left( \frac{1}{n-1}\sum_{i=1}^{n-1} f(\xi_i) \right)\right)
            - \EE \left( \alpha\left( \frac{1}{n(n-1)}\sum_{i=1}^{n-1} f(\xi_i)
                         - \frac{1}{n} f(\xi_n) \right) \right),
 \end{split}
 \end{align}
 yielding \eqref{help6}.
We note that the right hand side of the inequality \eqref{help5} may take the form $-\infty - \infty = -\infty$, implying that its left hand side is $-\infty$, as well.

(ii). Since strongly $\alpha$-convex functions are automatically convex
 (due to the nonnegativity of $\alpha$), by Theorem \ref{Thm_monotonicity},
 we readily have part (ii).
We also mention that part (i) implies part (ii) as well taking into account that  
 $\EE\Big( \alpha\Big( \frac{1}{n(n-1)}\sum_{i=1}^{n-1} f(\xi_i) - \frac{1}{n}f(\xi_n) \Big)\Big)\in[0,\infty]$,
 $n\in\NN\setminus\{1\}$.

(iii).
Suppose, in addition, that the assumptions (1)--(5) in part (iii) hold as well.
In view of part (ii) of Proposition \ref{Pro1}, we can see that $f(\xi_i)$, $i\in\NN$, are exchangeable, whence we get that $f(\xi_i)$, $i\in\NN$, are identically distributed. Hence we have that
 \begin{align}\label{help11}
   \EE\left( \frac{1}{n(n-1)}\sum_{i=1}^{n-1} f(\xi_i)
                      - \frac{1}{n}f(\xi_n) \right)
    = 0, \qquad n\in\NN\setminus\{1\}.
 \end{align}
We are now going to check that
 \begin{align}\label{help8}
    \PP\left( \frac{1}{n(n-1)}\sum_{i=1}^{n-1} f(\xi_i)
                      - \frac{1}{n}f(\xi_n) =0  \right)
     <1, \qquad n\in\NN\setminus\{1\}.
 \end{align}
Let $n\in\NN\setminus\{1\}$ be fixed. To the contrary, assume that the equality $\frac{1}{n(n-1)}\sum_{i=1}^{n-1} f(\xi_i)= \frac{1}{n} f(\xi_n)$ holds almost surely. 
Then, with the notation $\eta:=\frac{1}{n-1}\sum_{i=1}^{n-1} f(\xi_i)$, we have that $\eta = f(\xi_n)$ holds almost surely.
Then, by condition (2) (and also by (1)) and the exchangeability of $f(\xi_i)$, $i\in\NN$, we get that
\Eq{help30}{
  0&\geq \cov(f(\xi_1),f(\xi_2))
  =\frac{1}{n-1}\sum_{i=1}^{n-1}\cov(f(\xi_i),f(\xi_n))
  =\cov(\eta,f(\xi_n))\\
  &=\EE(\eta f(\xi_n))-\EE(\eta)\EE(f(\xi_n))
  =\EE((f(\xi_n))^2)-(\EE(f(\xi_n)))^2=\DD^2(f(\xi_n)).
}
On the other hand, by the assumption $(1)$ and the exchangeability of $f(\xi_i)$, $i\in\NN$, we have that $\DD^2(f(\xi_n))>0$, which contradicts the inequality \eqref{help30}.
This completes the proof of \eqref{help8}.

Using \eqref{help11}, \eqref{help8} and the strict convexity of $\alpha$ (assumption $(3)$),
 the Jensen inequality implies that
 \begin{align*}
  &\EE\left( \alpha\left( \frac{1}{n(n-1)}\sum_{i=1}^{n-1} f(\xi_i) - \frac{1}{n}f(\xi_n) \right)\right)\\
  &\qquad  > \alpha\left( \frac{1}{n(n-1)}\sum_{i=1}^{n-1} \EE(f(\xi_i)) - \frac{1}{n}\EE(f(\xi_n)) \right)
          =\alpha(0) =0, \qquad n\in\NN\setminus\{1\},
 \end{align*}
 where the last equality holds due to part (i) of Corollary \ref{Cor_MakNikPal1}.
Next, we check that the assumption $(4)$ yields that
 \begin{align}\label{help24}
 \EE\left( \alpha\left( \frac{1}{n(n-1)}\sum_{i=1}^{n-1} f(\xi_i) - \frac{1}{n}f(\xi_n) \right)\right)
   \in(0,\infty),\qquad  n\in\NN\setminus\{1\}.
 \end{align}
The convexity of $\alpha$ (assumption $(3)$) yields that
 \begin{align}\label{Rem_cond_e_elegseges}
  \begin{split}
  \alpha\left( \frac{1}{n(n-1)}\sum_{i=1}^{n-1} f(\xi_i) - \frac{1}{n}f(\xi_n) \right) 
  &  = \alpha\left( \frac{1}{n-1}\sum_{i=1}^{n-1} \frac{f(\xi_i) - f(\xi_n)}{n} \right) \\
  & \leq \frac{1}{n-1} \sum_{i=1}^{n-1} \alpha\left( \frac{f(\xi_i) - f(\xi_n)}{n} \right),
    \qquad n\in\NN\setminus\{1\}.
 \end{split}   
 \end{align} 
Since $(\xi_n)_{n\in\NN}$ are exchangeable, by part (ii) of Proposition \ref{Pro1},
 $(f(\xi_n))_{n\in\NN}$ are exchangeable as well, and therefore 
 $(f(\xi_1),f(\xi_2))$ and $(f(\xi_i),f(\xi_n))$, $i=1,\ldots,n-1$, have the same distributions, where $n\in\NN\setminus\{1\}$,
 yielding that
 \begin{align*}
   \EE \left( \alpha\left( \frac{f(\xi_i) - f(\xi_n)}{n} \right) \right)
      = \EE \left( \alpha\left( \frac{f(\xi_1) - f(\xi_2)}{n} \right) \right),
      \qquad i=1,\ldots,n-1,\quad n\in\NN\setminus\{1\}.
 \end{align*}
Consequently, by taking expectations of \eqref{Rem_cond_e_elegseges} side by side, we get that
  \begin{align}\label{help31}
  \begin{split}
  \EE\left(\alpha\left( \frac{1}{n(n-1)}\sum_{i=1}^{n-1} f(\xi_i) - \frac{1}{n}f(\xi_n) \right)\right)
  \leq \EE\left(\alpha\left( \frac{f(\xi_1) - f(\xi_2)}{n} \right)\right),
 \quad n\in\NN\setminus\{1\}.
 \end{split}
  \end{align}  
Using again the convexity of $\alpha$, we have that   
 \Eq{*}{
  \alpha\left(\! \frac{f(\xi_1)-f(\xi_2)}{n}\!\right)
   &=\alpha \left( \frac{2}{n}\cdot\frac{f(\xi_1)-f(\xi_2)}{2}+\frac{n-2}{n}\cdot 0\right)\\
   &\leq \frac{2}{n}\alpha \left(\! \frac{f(\xi_1)-f(\xi_2)}{2}\!\right) + \frac{n-2}{n}\alpha(0)
    = \frac{2}{n}\alpha \left(\! \frac{f(\xi_1)-f(\xi_2)}{2}\!\right),
    \quad n\in\NN\setminus\{1\},
   }
 where the last equality holds due to part (i) of Corollary \ref{Cor_MakNikPal1}.
Taking into account \eqref{help31} and the assumption $(4)$, we obtain that
 \[
   \EE\left(\alpha\left( \frac{1}{n(n-1)}\sum_{i=1}^{n-1} f(\xi_i) - \frac{1}{n}f(\xi_n) \right)\right) 
      \leq \frac{2}{n}  \EE\left(\alpha\left( \frac{f(\xi_1) - f(\xi_2)}{2} \right)\right)
       <\infty, 
 \]
 yielding \eqref{help24}, as desired.
 
To conclude, the statement of part (iii) follows by \eqref{help6} taking into account that all the three expectations involved in \eqref{help6} are finite, and the third one is positive (these facts are consequences of part (i), the assumption $(5)$ and \eqref{help24}).
\end{proof}

We note that if the error function $\alpha$ is of the form $\alpha(t):=\mu t^2$, $t\in f(I)-f(I)$, where $\mu\in\RR_+$,
 then the statement of part (i) of Theorem \ref{Thm_monotonicity_alpha} simplifies to that of Theorem \ref{Thm_strict_monotonicity} 
 under the additional assumption that $\DD^2(f(\xi_1))\in(0,\infty)$.
Indeed, by the paragraph after Definition \ref{Def_strongly_convex_2}, 
 a strongly convex function with such an error function is strongly convex with modulus $\mu$, and, by formulae (30) and (32) in Mattei and Garreau \cite{MatGar},
 the correction term on the right hand side of the inequality \eqref{help6} takes the form 
 $-\frac{\mu(1-\varrho)}{n(n-1)}\DD^2(f(\xi_1))\in(-\infty,0]$, $n\in\NN\setminus\{1\}$, where 
 $\varrho$ is the correlation coefficient of $f(\xi_1)$ and $f(\xi_2)$, yielding the statement of Theorem \ref{Thm_strict_monotonicity}
 (using that $\frac{1}{n(n-1)} = \frac{1}{n-1} - \frac{1}{n}$).

In the following remark, an alternative form and an estimation of the second term on the right hand side of the inequality \eqref{help6} can be found.

\begin{Rem}
(i). Suppose that the assumptions of Theorem \ref{Thm_monotonicity_alpha} hold.
The second term on the right hand side of the inequality \eqref{help6} can be written in another form, namely,
 \[
 \EE\left( \alpha\left( \frac{1}{n(n-1)}\sum_{i=1}^{n-1} f(\xi_i) - \frac{1}{n}f(\xi_n) \right)\right)
  =  \EE\left(\alpha\left(    \frac{1}{n-1} \left(  \frac{1}{n}\sum_{i=1}^n f(\xi_i) - f(\xi_k) \right) \right) \right)
 \]
 for each $n\in\NN\setminus\{1\}$ and $k\in\{1,\ldots,n\}$, since
 \begin{align*}
  \frac{1}{n(n-1)}\sum_{i=1}^{n-1} f(\xi_i) - \frac{1}{n}f(\xi_n) 
    = \frac{1}{n-1} \left(  \frac{1}{n}\sum_{i=1}^n f(\xi_i) - f(\xi_n) \right),\qquad n\in\NN\setminus\{1\},
 \end{align*}
 and the exchangeability of $(\xi_n)_{n\in\NN}$ yields that the distributions of 
  \[
     \frac{1}{n}\sum_{i=1}^n f(\xi_i) - f(\xi_n)
       = \frac{1}{n} \sum_{i=1}^{n-1} f(\xi_i) + \frac{1-n}{n} f(\xi_n),
  \]   
    and
  \[  
      \frac{1}{n}\sum_{i=1}^n f(\xi_i) - f(\xi_k)
        = \frac{1}{n} \sum_{i\in\{1,\ldots,n\}\setminus\{k\}} f(\xi_i) + \frac{1-n}{n} f(\xi_k) , \quad k\in\{1,\ldots,n\},
  \]
coincide (following from parts (ii) and (iii) of Proposition \ref{Pro1} with the choices $\lambda_i:=\frac{1}{n}$, $i=1,\ldots,n-1$, and $\lambda_n:=\frac{1-n}{n}$).
  
(ii). Under the assumptions of part (iii) of Theorem \ref{Thm_monotonicity_alpha},
 by its proof, the second term on the right hand side of the inequality \eqref{help6} can be upper estimated as follows:
  \[
   \EE\left(\alpha\left( \frac{1}{n(n-1)}\sum_{i=1}^{n-1} f(\xi_i) - \frac{1}{n}f(\xi_n) \right)\right) 
      \leq \frac{2}{n}  \EE\left(\alpha\left( \frac{f(\xi_1) - f(\xi_2)}{2} \right)\right),
       \qquad n\in\NN\setminus\{1\}.
 \]
\proofend
\end{Rem}

In the following remark, we provide a sufficient condition under which  the assumption $(4)$ in part (iii) of Theorem \ref{Thm_monotonicity_alpha} holds.

\begin{Rem}
Suppose that the assumptions of Theorem \ref{Thm_monotonicity_alpha} hold.
If, in addition, $\alpha$ is convex and $\EE(\widehat\alpha(f(\xi_1) - f(\xi_2)))<\infty$ (where $\widehat\alpha$ 
  is defined in \eqref{def_alphahat}), then the assumption $(4)$ in part (iii) of Theorem \ref{Thm_monotonicity_alpha} holds.
Indeed, using part (i) of Remark \ref{Rem_starlikeness}, $\alpha(0)=0$ (following from part (i) of Corollary \ref{Cor_MakNikPal1}) 
 and the inequality \eqref{uv}, we get that 
 \[
 \EE\left(\alpha\left( \frac{f(\xi_1) - f(\xi_2)}{2} \right)\right)
  \leq \EE\left(\widehat\alpha\left( \frac{f(\xi_1) - f(\xi_2)}{2} \right)\right)
  \leq \frac{1}{2} \EE\left(\widehat\alpha\left( f(\xi_1) - f(\xi_2) \right)\right)
  <\infty, 
 \]
 as desired.
 \proofend
 \end{Rem}

The following remark highlights the assumption $(2)$ in part (iii) of Theorem \ref{Thm_monotonicity_alpha} 
 in connection with exchangeability.

\begin{Rem}\label{Rem_condiii}
Under the assumptions of part (iii) of Theorem \ref{Thm_monotonicity_alpha}, the sequence $f(\xi_i)$, $i\in\NN$, is exchangeable (hence identically distributed) and pairwise nonpositively correlated.
However, in general, it does not hold that $f(\xi_i)$, $i\in\NN$, are independent, as the following counterexample shows (where, for simplicity, instead of $f(\xi_i)$, we write $\eta_i$).
Let $\be_1,\be_2,\be_3$ denote the natural basis in $\RR^3$, and let $(\eta_1,\eta_2,\eta_3)$ be a $3$-dimensional random vector such
 \[
   \PP((\eta_1,\eta_2,\eta_3) = \be_k) = \PP((\eta_1,\eta_2,\eta_3) = -\be_k) = \frac{1}{6},\qquad k\in\{1,2,3\}.
 \]
Then 
 \[
   \EE(\eta_i) = 1\cdot \frac{1}{6} + (-1)\cdot \frac{1}{6} + 0 \cdot\frac{4}{6} = 0,\qquad i\in\{1,2,3\},
 \]
 and, since $\eta_i\eta_j=0$ for $i\ne j$, $i,j\in\{1,2,3\}$, we have that 
 \[
  \cov(\eta_i,\eta_j) = \EE(\eta_i\eta_j) -  \EE(\eta_i)\EE(\eta_j) = 0 - 0\cdot 0=0, \qquad i\ne j, \quad i,j\in\{1,2,3\}.
 \]
Therefore, $\eta_1$, $\eta_2$ and $\eta_3$ are pairwise uncorrelated.
On the other hand, $\eta_1$, $\eta_2$ and $\eta_3$ are not independent, since, for example, 
 \[
   \PP(\eta_1=1)=\frac{1}{6},\qquad  \PP(\eta_2=1)=\frac{1}{6},\qquad \PP(\eta_1=1,\eta_2=1)=\PP(\emptyset)=0,
 \]
 yielding that $\PP(\eta_1=1,\eta_2=1)\ne \PP(\eta_1=1)\PP(\eta_2=1)$.
 
Finally, we check that $\eta_1$, $\eta_2$ and $\eta_3$  are exchangeable. For this, we are going to show that, for any permutation $\sigma$ of $(1,2,3)$, the distributions of $(\eta_1,\eta_2,\eta_3)$ and $(\eta_{\sigma(1)},\eta_{\sigma(2)},\eta_{\sigma(3)})$ coincide.
Let $\sigma$ be an arbitrary permutation of $(1,2,3)$. Let $S$ denote the set $\big\{\be_1,\be_2,\be_3,-\be_1,-\be_2,-\be_3\big\}$. Then, for any $(x_1,x_2,x_3)\in\RR^3$, the relation $(x_1,x_2,x_3)\in S$ holds if and only if $(x_{\sigma^{-1}(1)},x_{\sigma^{-1}(2)},x_{\sigma^{-1}(3)})\in S$, where $\sigma^{-1}$ denotes the inverse of $\sigma$.
Hence 
 \begin{align*}
 \PP( (\eta_1,\eta_2,\eta_3) = (x_1,x_2,x_3))
     = \begin{cases}
         \frac{1}{6} & \text{if $(x_1,x_2,x_3)\in S$,}\\
         0  & \text{if $(x_1,x_2,x_3)\in \RR^3\setminus S$.}
      \end{cases}
 \end{align*}
 and 
 \begin{align*}
 \PP( (\eta_{\sigma(1)},\eta_{\sigma(2)},\eta_{\sigma(3)}) = (x_1,x_2,x_3))
  & = \PP( (\eta_1,\eta_2,\eta_3) = (x_{\sigma^{-1}(1)},x_{\sigma^{-1}(3)},x_{\sigma^{-1}(3)})) \\
   & = \begin{cases}
         \frac{1}{6} & \text{if $(x_{\sigma^{-1}(1)},x_{\sigma^{-1}(2)},x_{\sigma^{-1}(3)})\in S$,}\\
         0  & \text{if $(x_{\sigma^{-1}(1)},x_{\sigma^{-1}(2)},x_{\sigma^{-1}(3)})\in \RR^3\setminus S$}
      \end{cases}\\
  & = \begin{cases}
         \frac{1}{6} & \text{if $(x_1,x_2,x_3)\in S$,}\\
         0  & \text{if $(x_1,x_2,x_3)\in \RR^3\setminus S$.}
      \end{cases} 
 \end{align*}
Hence 
 \[
   \PP( (\eta_1,\eta_2,\eta_3) = (x_1,x_2,x_3)) = \PP( (\eta_{\sigma(1)},\eta_{\sigma(2)},\eta_{\sigma(3)}) = (x_1,x_2,x_3))
 \] 
 for any $(x_1,x_2,x_3)\in\RR^3$, that is, $(\eta_1,\eta_2,\eta_3)$ and $(\eta_{\sigma(1)},\eta_{\sigma(2)},\eta_{\sigma(3)})$
 have the same distributions, as desired.
\proofend
\end{Rem}

In the next remark, we shed some light on the role of the assumptions (4) and (5) in part (iii) of Theorem \ref{Thm_monotonicity_alpha}.

\begin{Rem}\label{Rem_condiv_v}
The assumptions (4) and (5) in part (iii) of Theorem \ref{Thm_monotonicity_alpha}
 are in fact necessary as well in order that the sequence $\Big(\EE\big(L\big( \mathscr{A}_n^{f}(\xi_1,\ldots,\xi_n) \big)\big)\Big)_{n\in\NN}$ be strictly decreasing.
Indeed, the inequality \eqref{help6} with $n=2$ yields that 
 \[
 \EE\big(L\big( \mathscr{A}_2^{f}(\xi_1,\xi_2) \big)\big)
   \leq \EE(L(\xi_1)) - \EE\left( \alpha\left( \frac{1}{2}(f(\xi_1) - f(\xi_2)) \right)\right),
 \]
 where, by the assumption, $\EE(L(\xi_1))\in\RR$.
If $\EE\left( \alpha\left( \frac{1}{2}(f(\xi_1) - f(\xi_2)) \right)\right)$ were $\infty$, then 
 $\EE\big(L\big( \mathscr{A}_2^{f}(\xi_1,\xi_2) \big)\big)$ would be $-\infty$, yielding that 
 $\EE\big(L\big( \mathscr{A}_n^{f}(\xi_1,\ldots,\xi_n) \big)\big)$ would be $-\infty$ for all $n\geq 3$
 as well (following from \eqref{help6}).
Similarly, if $\EE\big(L\big( \mathscr{A}_{n_0}^{f}(\xi_1,\ldots,\xi_{n_0}) \big)\big)$ were $-\infty$
 for some $n_0\in\NN$, then, by  \eqref{help6}, $\EE\big(L\big( \mathscr{A}_n^{f}(\xi_1,\ldots,\xi_n) \big)\big)$ 
 would be $-\infty$ for all $n\geq n_0$ as well.
\proofend
\end{Rem}

In the next corollary, we specialize part (i) of Theorem \ref{Thm_monotonicity_alpha} to the case $\alpha(x):=x^4$, $x\in\RR$, and for a sequence of independent and identically distributed random variables $(\xi_n)_{n\in\NN}$.

\begin{Cor}\label{Cor_x4}
Let $I\subseteq \RR$ be a nonempty open interval, $f:I\to\RR$ be a continuous and strictly increasing function, and let $(\xi_n)_{n\in\NN}$ be a sequence of independent and identically distributed random variables with values in $I$
 such that $\EE((f(\xi_1))^4)<\infty$.
Further, let $L:I\to\RR$ be a function such that the function $L\circ f^{-1}: f(I)\to\RR$ is strongly $\alpha$-convex,
where $\alpha:(f(I)-f(I))\to\RR_+$, $\alpha(x):=x^4$, $x\in (f(I)-f(I))$, and $\EE(\vert L(\xi_1)\vert)<\infty$.
Then
 \begin{itemize}
  \item[(i)] $\EE\big(L\big( \mathscr{A}_n^{f}(\xi_1,\ldots,\xi_n) \big)\big)\in[-\infty,\infty)$, $n\in\NN\setminus\{1\}$,
  \item[(ii)] for each $n\in\NN\setminus\{1\}$,
          \begin{align*}
               &\EE\big(L\big( \mathscr{A}_n^{f}(\xi_1,\ldots,\xi_n) \big)\big)\\
               &\qquad \leq \EE\big(L\big( \mathscr{A}_{n-1}^{f}(\xi_1,\ldots,\xi_{n-1}) \big)\big) 
                 - \frac{1}{(n-1)^3n^3} \Big( (n^2-3n+3) M_4 + 3(2n-3) M_2^2 \Big),
           \end{align*}
           where $M_k := \EE ( ( f(\xi_1) - \EE(f(\xi_1)) )^k )$, $k\in\{2,4\}$.
 \end{itemize} 
\end{Cor} 

Before proving Corollary \ref{Cor_x4}, in the next remark, we point out the fact that
in some cases the correction term on the right hand side of the inequality in part (ii)
of Corollary \ref{Cor_x4} can be written in a simpler form, 
and we also present a weaker form the inequality in question.

\begin{Rem}
(i) In part (ii) of Corollary \ref{Cor_x4}, in case of $n=2$ and $n=3$, the correction term
 \begin{align}\label{help_Cor_identical}
     \frac{1}{(n-1)^3n^3} \Bigg[ (n^2-3n+3) M_4 + 3(2n-3) M_2^2  \Bigg]
 \end{align}    
 can be written as 
 \[
  \frac{1}{2^4} \EE((f(\xi_1)-f(\xi_2))^4) \qquad \text{and}\qquad \frac{1}{2^4 3^2} \EE((f(\xi_1)-f(\xi_2))^4),
 \]
 respectively.
Indeed, since $f(\xi_1)$ and $f(\xi_2)$ are independent and identically distributed, we have $\EE(f(\xi_1))=\EE(f(\xi_2))$, and hence, 
 by the multinomial theorem, 
 \begin{align*}
   \EE((f(\xi_1)&-f(\xi_2))^4)
    =\EE\Big([f(\xi_1)-\EE(f(\xi_1))]-[f(\xi_2)-\EE(f(\xi_2))])^4\Big)\\
   &= \EE([f(\xi_1)-\EE(f(\xi_1))]^4)
      + 4 \EE\big([f(\xi_1)-\EE(f(\xi_1))]^3[f(\xi_2)-\EE(f(\xi_2))]\big)\\
   &\phantom{=\;} + 6\EE\big([f(\xi_1)-\EE(f(\xi_1))]^2[f(\xi_2)-\EE(f(\xi_2))]^2\big)\\
   &\phantom{=\;} + 4 \EE\big([f(\xi_1)-\EE(f(\xi_1))][f(\xi_2)-\EE(f(\xi_2))]^3\big)
                  + \EE([f(\xi_2)-\EE(f(\xi_2))]^4)\\   
   &=2\EE([f(\xi_1)-\EE(f(\xi_1))]^4)+6\EE\big([f(\xi_1)-\EE(f(\xi_1))]^2\big)\EE\big([f(\xi_2)-\EE(f(\xi_2))]^2\big)\\
   &=2\EE([f(\xi_1)-\EE(f(\xi_1))]^4)+6(\DD^2(f(\xi_1)))^2
    =2 M_4+6 M_2^2.
 \end{align*}
Note also that for $n\geq 4$, $n\in\NN$, it does not hold that the correction term \eqref{help_Cor_identical} is a 
constant multiple of $\EE((f(\xi_1)-f(\xi_2))^4)$, since the equation $3(n^2-3n+3)=3(2n-3)$ has only two solutions $2$ and $3$.

(ii) Under the conditions of Corollary \ref{Cor_x4}, since $M_4\geq M_2^2$ (due to Lyapunov's inequality), we also 
 have that 
 \begin{align*}
  \EE\big(L\big( \mathscr{A}_n^{f}(\xi_1,\ldots,\xi_n) \big)\big)
   \leq \EE\big(L\big( \mathscr{A}_{n-1}^{f}(\xi_1,\ldots,\xi_{n-1}) \big)\big) 
                - \frac{n^2+3n-6}{(n-1)^3n^3} M_2^2
 \end{align*}
 for each $n\in\NN\setminus\{1\}$, where $n^2+3n-6>0$ for each $n\in\NN\setminus\{1\}$. 
In particular, this inequality shows that the sequence $(\EE\big(L\big( \mathscr{A}_n^{f}(\xi_1,\ldots,\xi_n) \big)\big))_{n\in\NN}$ 
 is strictly decreasing whenever $M_2\neq0$ (i.e., $f(\xi_1)$ is not a constant almost surely) 
 and the assumption (5) of part (iii) in Theorem \ref{Thm_monotonicity_alpha} holds.
\proofend
\end{Rem}

\noindent{\bf Proof of Corollary \ref{Cor_x4}}.
(i).
It readily follows from part (i) of Theorem \ref{Thm_monotonicity_alpha}.

(ii). We evaluate the term
 \[
  \EE\left( \alpha\left( \frac{1}{n(n-1)}\sum_{i=1}^{n-1} f(\xi_i) - \frac{1}{n}f(\xi_n) \right)\right),
    \qquad n\in\NN\setminus\{1\},
 \]
 in part (i) of Theorem \ref{Thm_monotonicity_alpha} for the function 
 $\alpha:(f(I)-f(I))\to\RR_+$, $\alpha(x)=x^4$, $x\in (f(I)-f(I))$.

Let $\eta_\ell:=f(\xi_\ell)-\EE(f(\xi_\ell))$, $\ell\in\NN$.
Since $f(\xi_\ell)$, $\ell\in\NN$,  are independent and identically distributed, we get that
  $\eta_\ell$, $\ell\in\NN$, are independent and identically distributed, and 
  $\EE(\eta_\ell)=0$, $\ell\in\NN$.
Hence, for each $n\in\NN\setminus\{1\}$, we obtain that
 \begin{align*}
  &\EE\left( \left( \frac{1}{n(n-1)}\sum_{i=1}^{n-1} f(\xi_i) - \frac{1}{n}f(\xi_n) \right)^4  \right)
   = \frac{1}{(n-1)^4 n^4}
      \EE\left( \left( \sum_{i=1}^{n-1} \eta_i -  (n-1) \eta_n \right)^4  \right)\\
   & = \frac{1}{(n-1)^4 n^4}
       \EE\Bigg[  \left( \sum_{i=1}^{n-1} \eta_i \right)^4
      -4(n-1)\left( \sum_{i=1}^{n-1} \eta_i \right)^3\eta_n
   + 6(n-1)^2 \left( \sum_{i=1}^{n-1} \eta_i \right)^2\eta_n^2 \\
   &\phantom{= \frac{1}{(n-1)^4 n^4} \EE\Bigg[}
       - 4(n-1)^3 \left( \sum_{i=1}^{n-1} \eta_i \right) \eta_n^3 +(n-1)^4 \eta_n^4
       \Bigg]  \\
   & = \frac{1}{(n-1)^4 n^4}
       \Bigg[ \EE\left( \left( \sum_{i=1}^{n-1} \eta_i \right)^4 \right)+ 6(n-1)^2 \EE\left( \left( \sum_{i=1}^{n-1} \eta_i \right)^2\right) \EE ( \eta_n^2 ) + (n-1)^4 \EE ( \eta_n^4 )
       \Bigg].
 \end{align*}
Here we get that
 \begin{align*}
  \EE\left( \left( \sum_{i=1}^{n-1} \eta_i \right)^2 \right)
   & = \sum_{i=1}^{n-1} \EE( \eta_i^2 )
    = (n-1)\EE( \eta_1^2 ),
 \end{align*}
 and, by the multinomial theorem, for each $n\in\NN\setminus\{1\}$,
 \begin{align*}
   \EE\left( \left( \sum_{i=1}^{n-1} \eta_i \right)^4 \right)
   &= \sum_{\substack{\ell_1+\cdots+\ell_{n-1}=4,\\[1mm]  \ell_1,\ldots,\ell_{n-1}\in\ZZ_+\setminus\{1\}}}
        \frac{4!}{\ell_1!\cdots\ell_{n-1}!}
         \EE\big( \eta_1^{\ell_1} \big)\cdots \EE\big( \eta_1^{\ell_{n-1}} \big)\\
  &= \binom{n-1}{2}\frac{4!}{2!2!}  \big(\EE\big( \eta_1^2 \big)\big)^2
             +\binom{n-1}{1} \frac{4!}{4!}  \EE\big( \eta_1^4 \big)\\
 &= 3(n-1)(n-2) (\EE( \eta_1^2 ))^2 + (n-1)\EE\big( \eta_1^4 \big).
 \end{align*}
Consequently, we get that
 \begin{align*}
  &\EE\left( \left( \frac{1}{n(n-1)}\sum_{i=1}^{n-1} f(\xi_i) - \frac{1}{n}f(\xi_n) \right)^4  \right)\\
  &\qquad = \frac{1}{(n-1)^4n^4} \Bigg[ 3(n-1)(n-2) (\EE( \eta_1^2 ))^2 + (n-1)\EE\big( \eta_1^4 \big) \\
  &\phantom{\qquad = \frac{1}{(n-1)^4n^4} \Bigg[}
          + 6(n-1)^3  (\EE( \eta_1^2 ))^2 + (n-1)^4 \EE ( \eta_n^4 )
              \Bigg] \\
 &\qquad = \frac{1}{(n-1)^3n^4} \Bigg[ (1 + (n-1)^3) \EE ( \eta_1^4 ) + 3n(2n-3)(\EE( \eta_1^2 ))^2  \Bigg]\\
 &\qquad= \frac{1}{(n-1)^3n^3} \Bigg[ (n^2-3n+3) M_4 + 3(2n-3) M_2^2  \Bigg].
 \end{align*}
Therefore, the statement of part (ii) follows by part (i) of Theorem \ref{Thm_monotonicity_alpha}.
\proofend

In the next example, we specialize Corollary \ref{Cor_x4} to the case of geometric mean and a sequence of i.i.d.\ random variables, and, in the second part of this example, we make a further specialization by choosing the uniform distribution on $(0,1)$ as the common distribution of the i.i.d.\ sequence of random variables in question.

\begin{Ex}\label{Ex_Sect5_uj}
Let $I:=(0,\infty)$, $f:I\to\RR$, $f(x):=\ln(x)$, $x\in I$, and
 $(\xi_n)_{n\in\NN}$ be a sequence of independent and identically distributed positive random variables
 such that $\EE((\ln(\xi_1))^4)<\infty$.
Then $f(I)=\RR$.
Let $L:I\to\RR$, $L(x):=(\ln(x))^4$, $x\in I$, and $\alpha:\RR\to\RR_+$, $\alpha(t):=t^4$, $t\in\RR$.
Then $\EE(\vert L(\xi_1)\vert) = \EE((\ln(\xi_1))^4)<\infty$, and, by Example \ref{Ex_power_strongly_convex}, 
 the function $\RR\ni x\mapsto (L\circ f^{-1})(x) = x^4$ is strongly $\alpha$-convex, and   
 \[
   \cA_n^f(\xi_1,\ldots,\xi_n) = (\xi_1\cdots\xi_n)^{\frac{1}{n}},\qquad n\in\NN.
 \]
By Corollary \ref{Cor_x4}, we have that
 \[
    \EE\Big[ \big( \ln\big( (\xi_1\cdots\xi_n)^{\frac{1}{n}} \big) \big)^4 \Big]
     = \frac{1}{n^4} \EE\left(  \left(  \sum_{i=1}^n \ln(\xi_i) \right)^4  \right)
     <\infty, \qquad n\in\NN,
 \]
 and 
 \begin{align}\label{help_strong_alpha_pelda_1}
 \begin{split}
   &\frac{1}{n^4} \EE\left(  \left(  \sum_{i=1}^n \ln(\xi_i) \right)^4  \right)\\
   &\quad\leq \frac{1}{(n-1)^4} \EE\left(  \left(  \sum_{i=1}^{n-1} \ln(\xi_i) \right)^4  \right) \\
   &\phantom{\quad\leq} - \frac{1}{(n-1)^3n^3} \Bigg[ (n^2 - 3n +3) \EE ( ( \ln(\xi_1) - \EE(\ln(\xi_1)) )^4 ) 
                     + 3(2n-3)(\DD^2(\ln(\xi_1)))^2  \Bigg]
 \end{split}                    
 \end{align}
 for each $n\in\NN\setminus\{1\}$.

Next, we specialize the inequality \eqref{help_strong_alpha_pelda_1} when $\xi_1$ is uniformly distributed on the interval $(0,1)$. 
In Example \ref{Ex_Sect5}, we checked that $-\ln(\xi_1)$ is exponentially distributed with parameter 1,
 and $\sum_{i=1}^n (-\ln(\xi_i))$, $n\in\NN$, is Gamma distributed with parameters $n$ and $1$.
Hence we have that
  \[
    \EE\left(\left( \sum_{i=1}^n (-\ln(\xi_i)) \right)^4\right)
        = \frac{\Gamma(n+4)}{\Gamma(n)} 
        = n(n+1)(n+2)(n+3),\qquad n\in\NN,
 \]  
 where $\Gamma$ denotes the Gamma function, and, using also \eqref{help25},
 \begin{align*}
  \EE ( ( \ln(\xi_1) - \EE(\ln(\xi_1)) )^4 )
    &= \EE ( ( \ln(\xi_1) + 1 )^4 )\\
    &= \EE ( ( \ln(\xi_1) )^4 ) + 4\EE ( ( \ln(\xi_1) )^3 ) + 6 \EE ( ( \ln(\xi_1) )^2 ) + 4 \EE ( \ln(\xi_1) ) + 1 \\
    &= 4! - 4\cdot 3! + 6\cdot 2! -4+1
    = 9.
 \end{align*} 
Therefore, in the considered special case, using also \eqref{help_Var_Exp}, the inequality \eqref{help_strong_alpha_pelda_1} takes the form
 \begin{align*}
  \frac{(n+1)(n+2)(n+3)}{n^3}
     & \leq \frac{n(n+1)(n+2)}{(n-1)^3}
       - \frac{1}{(n-1)^3n^3} \Big( 9(n^2 -3n+3) + 3(2n-3)\Big) \\
     & = \frac{n(n+1)(n+2)}{(n-1)^3}
        - \frac{1}{(n-1)^3n^3} \Big( 9n^2 -27n + 27 + 6n -9\Big)\\
     &  = \frac{n(n+1)(n+2)}{(n-1)^3}
           - \frac{3}{(n-1)^3n^3}(3n^2 -7n +6),
           \qquad n\in\NN\setminus\{1\}.
 \end{align*}
This inequality can also be verified directly. Indeed, multiplying both sides of the previous inequality by $(n-1)^3n^3$, we arrive at
  \[
    (n-1)^3(n+1)(n+2)(n+3) \leq n^4(n+1)(n+2) - 3(3n^2-7n+6), \qquad n\in\NN\setminus\{1\}.
  \]
By some algebraic calculations, we can conclude that the inequality \eqref{help_strong_alpha_pelda_1}
 is equivalent to 
 \begin{align*}
  6n^4 + 10n^3 -18n^2 + 14n -12  
  = 2(n-1)(3n^3 + 7n^2 + n(n-1) +6)\geq 0, \qquad n\in\NN\setminus\{1\},
 \end{align*}
 which is satisfied, since all the three terms of the product on the left hand side are positive for each $n\in\NN\setminus\{1\}$.
\proofend
\end{Ex}

\section*{Data availability statements}
Data sharing is not applicable to this article as no datasets were generated or analyzed during the current study.

\section*{Declaration of competing interest}
The authors declare that they have no conflicts of interest.

\bibliographystyle{plain}

\end{document}